\documentclass[11pt,a4paper]{amsart}

\usepackage[T1]{fontenc}
\usepackage[utf8]{inputenc}
\usepackage[margin=1in]{geometry}

\usepackage{nicematrix}
\usepackage{amsfonts,amssymb,amsthm,amsmath,amstext}
\usepackage{dsfont} 
\usepackage{mathtools}
\usepackage{colonequals}

\usepackage[shortlabels]{enumitem}
\usepackage[skip=1ex]{caption}
\usepackage{tikz}
\usepackage{tikz-cd}
\usepackage{xcolor} 

\usepackage{csquotes}
\usepackage[
  style=alphabetic, 
  backend=biber, 
  sorting=nyt, 
  doi=false, 
  isbn=false,
  url=false, 
  hyperref,
  backref,
  backrefstyle=none, 
  maxnames=20, 
  maxalphanames=20
]{biblatex}
\newbibmacro{string+doiurlisbn}[1]{%
  \iffieldundef{doi}{%
    \iffieldundef{url}{%
      \iffieldundef{isbn}{%
        \iffieldundef{issn}{%
          #1%
        }{%
          \href{https://books.google.com/books?vid=ISSN\thefield{issn}}{#1}%
        }%
      }{%
        \href{https://books.google.com/books?vid=ISBN\thefield{isbn}}{#1}%
      }%
    }{%
      \href{\thefield{url}}{#1}%
    }%
  }{%
    \href{https://dx.doi.org/\thefield{doi}}{#1}%
  }%
}
\DeclareFieldFormat{title}{\usebibmacro{string+doiurlisbn}{\mkbibemph{#1}}}
\DeclareFieldFormat[article,incollection,inproceedings,unpublished,misc,book]{title}%
    {\usebibmacro{string+doiurlisbn}{\mkbibquote{#1}}}
\DefineBibliographyStrings{english}{%
  backrefpage = {$\uparrow$},%
  backrefpages = {$\uparrow$}%
}

\usepackage[pdfencoding=auto, psdextra, breaklinks=true]{hyperref}
\definecolor{mylinkcolor}{rgb}{0.5,0.0,0.0}
\definecolor{myurlcolor}{rgb}{0.0,0.0,0.7}
\hypersetup{
  colorlinks,
  urlcolor=myurlcolor,
  citecolor=myurlcolor,
  linkcolor=mylinkcolor,
  breaklinks=true,
}
\usepackage{xurl}

\numberwithin{equation}{section}

\usepackage[capitalize]{cleveref}

\usepackage{aliascnt}
\newcommand{\sharedthm}[2]{%
  \newaliascnt{#1}{equation}%
  \newtheorem{#1}[#1]{#2}%
  \aliascntresetthe{#1}}
\theoremstyle{plain}
\sharedthm{thm}{Theorem}
\sharedthm{prop}{Proposition}
\sharedthm{lem}{Lemma}
\sharedthm{cor}{Corollary}
\sharedthm{conj}{Conjecture}

\theoremstyle{definition}
\sharedthm{defn}{Definition}
\sharedthm{ex}{Example}
\sharedthm{remark}{Remark}
\sharedthm{question}{Question}
\sharedthm{algorithm}{Algorithm}

\crefalias{thm}{theorem}
\crefalias{prop}{proposition}
\crefalias{lem}{lemma}
\crefalias{cor}{corollary}
\crefalias{conj}{conjecture}
\crefalias{defn}{definition}
\crefalias{ex}{example}
\crefalias{remark}{remark}
\crefalias{question}{question}
\crefalias{algorithm}{algorithm}

\crefname{theorem}{Theorem}{Theorems}
\Crefname{theorem}{Theorem}{Theorems}

\crefname{proposition}{Proposition}{Propositions}
\Crefname{proposition}{Proposition}{Propositions}

\crefname{lemma}{Lemma}{Lemmas}
\Crefname{lemma}{Lemma}{Lemmas}

\crefname{corollary}{Corollary}{Corollaries}
\Crefname{corollary}{Corollary}{Corollaries}

\crefname{conjecture}{Conjecture}{Conjectures}
\Crefname{conjecture}{Conjecture}{Conjectures}

\crefname{definition}{Definition}{Definitions}
\Crefname{definition}{Definition}{Definitions}

\crefname{example}{Example}{Examples}
\Crefname{example}{Example}{Examples}

\crefname{remark}{Remark}{Remarks}
\Crefname{remark}{Remark}{Remarks}

\crefname{question}{Question}{Questions}
\Crefname{question}{Question}{Questions}

\crefname{algorithm}{Algorithm}{Algorithms}
\Crefname{algorithm}{Algorithm}{Algorithms}

\crefname{enumi}{step}{steps}
\Crefname{enumi}{Step}{Steps}
\crefname{hypothesis}{Hypothesis}{Hypotheses}
\Crefname{hypothesis}{Hypothesis}{Hypotheses}

\DeclareMathOperator{\Gal}{Gal}
\DeclareMathOperator{\Jac}{Jac}

\DeclareMathOperator{\Frob}{Frob}

\DeclareMathOperator{\Hom}{Hom}

\DeclareMathOperator{\Norm}{Nm}

\DeclareMathOperator{\res}{res}
\DeclareMathOperator{\GL}{GL}
\DeclareMathOperator{\GSp}{GSp}

\DeclareMathOperator{\SL}{SL}
\DeclareMathOperator{\PSL}{PSL}
\DeclareMathOperator{\PGL}{PGL}

\DeclareMathOperator{\tr}{tr}

\DeclareMathOperator{\divisor}{div}
\DeclareMathOperator{\Res}{Res}
\DeclareMathOperator{\Ind}{Ind}
\DeclareMathOperator{\Art}{Art}
\DeclareMathOperator{\Ver}{Ver}

\newcommand{\ab}{\mathrm{ab}}
\newcommand{\Gab}{G^\ab}

\newcommand{\A}{\mathbf{A}}
\newcommand{\C}{\mathbb{C}}
\newcommand{\F}{\mathbb{F}}

\newcommand{\Q}{\mathbb{Q}}
\newcommand{\R}{\mathbb{R}}
\newcommand{\Z}{\mathbb{Z}}

\newcommand{\Fbar}{\overline{\F}}
\newcommand{\Qbar}{\overline{\Q}}
\newcommand{\Kbar}{\overline{K}}
\newcommand{\OK}{\mathcal{O}_K}

\newcommand{\legendre}[2]{\ensuremath{\left( \frac{#1}{#2} \right)}}

\newcommand{\gl}[2]{\operatorname{GL}_{#1}(#2)}
\newcommand{\Fell}{\F_\ell}
\newcommand{\Fellbar}{\Fbar_\ell}
\newcommand{\pp}{\mathfrak{p}}
\newcommand{\ppbar}{\overline{\pp}}

\newcommand{\lam}{\lambda}
\newcommand{\lambar}{\overline{\lambda}}
\newcommand{\rhoell}{\rho_\ell}
\newcommand{\rhoellK}{\rho_\ell|_{G_K}}
\newcommand{\rhoellKss}{\rho_\ell|_{G_K}^{ss}}
\newcommand{\Id}{\mathrm{Id}}
\newcommand{\PP}{\mathbb{P}}

\title[On reducible $\ell$-torsion representations of Picard curves]
{On reducible $\ell$-torsion representations of Picard curves}

\author{Xiyao Chen}
\address{Nanyang Technological University, Singapore 639798}
\email{xiyaochen2002@gmail.com}

\author{Shiva Chidambaram}
\address{University of Wisconsin-Madison, Madison, WI 53706 USA.}
\email{chidambaram3@wisc.edu}

\author{Andrew Coursin}
\address{}
\email{coursin.andrew@gmail.com}

\author{Erik Reed}
\address{}
\email{erik.reed1983@gmail.com}

\author{Leo Yang}
\address{University of Wisconsin-Madison, Madison, WI 53706 USA.}
\email{yxj555595@gmail.com}


\date{}

\begin{document}
 
\maketitle

\begin{abstract}
Let $C/\mathbb{Q}$ be a Picard curve and let $\ell\equiv1\pmod3$ be a 
prime such that the $\lambda$-torsion representation of the Jacobian 
of $C$ is reducible for a prime $\lambda$ of $\mathbb{Q}(\zeta_3)$ 
above $\ell$. We describe the constraints imposed by complex 
conjugation on the $2$-dimensional subquotient $\rho$, in 
particular, showing that if it is absolutely irreducible, then up 
to twist it descends to $\mathbb{Q}$.
We then give criteria for ruling out the irreducible case, using
Dickson's classification of subgroups of $\GL_2(\Fell)$, under the 
natural hypothesis that the characteristic polynomials of all elements
in the image of $\rho$ split.
We also develop algorithms to compute the set of $1$-dimensional 
Jordan--H\"older constitutents of the $\ell$-torsion representation 
over $\mathbb{Q}(\zeta_3)$, and Frobenius characteristic polynomials 
of the $2$-dimensional subquotient $\rho$.
We apply these methods to $381$ Picard curves with possibly
reducible $7$-torsion representations, and to five corresponding
$13$-torsion examples. Under the reducibility and splitting
hypotheses, we rule out an irreducible $2$-dimensional subquotient 
in all of the examples considered, showing that the semisimplification
of the $\ell$-torsion representation is a direct sum of $6$ characters.
\end{abstract}

\section{Introduction}
\label{sec:intro}

Let $K=\Q(\zeta_3)$, where $\zeta_3$ is a primitive cube root of
unity.  In this paper we study $\ell$-torsion Galois representations arising
from the Jacobians of Picard curves
\[
    C\colon y^3=f_4(x)
\]
defined over $\Q$, where $f_4$ is a separable quartic polynomial.
The map $(x,y)\mapsto (x,\zeta_3y)$ defines an order-$3$ automorphism
of $C$ and gives an action of $\Z[\zeta_3]$ on $\Jac(C)$ over $K$.  If
$\ell\equiv1\pmod 3$ is a rational prime, then $\ell=\lambda\overline{\lambda}$ in
$\mathcal O_K$, and the restriction to $G_K$ of the $\ell$-torsion
representation of $\Jac(C)$ decomposes into two $3$-dimensional
representations on the $\lambda$- and $\overline{\lambda}$-torsion.

We are interested in the case where the $\lambda$-torsion
representation $\rho_\lambda$ is reducible.  After interchanging
$\lambda$ and $\overline{\lambda}$ if necessary, we may suppose that
$\rho_\lambda$ admits a character $\varepsilon$ as a subrepresentation, and write
$\rho$ for the $2$-dimensional quotient. The main purpose of the paper is to
develop theoretical and computational tools for 
accessing the Frobenius characteristic polynomials of $\rho$
and determining whether
$\rho$ is irreducible, and, more generally, for recovering the
Jordan--H\"older constituents of the full $\ell$-torsion representation.

While we know in general that the $\lam$- and $\lambar$-torsions are symplectic duals with respect to the Weil pairing on $\Jac(C)[\ell]$, the fact that $C$ is defined over $\Q$ imposes strong additional
constraints.  Complex conjugation interchanges the
$\lambda$- and $\overline{\lambda}$-torsion representations.  If
$\rho$ is irreducible, this gives
\[
    \rho_\lam \simeq \varepsilon \oplus \rho,
    \qquad \text{and} \qquad 
    \varepsilon\varepsilon^c=\chi_\ell,
    \qquad
    \rho^c\simeq\chi_\ell\otimes\rho^{-T},
\]
where $\varepsilon^c, \rho^c$ denote the conjugate representations of 
$\varepsilon$ and $\rho$ respectively (see \cref{prop:detrho_conjugatedual}).
We show in \cref{prop:descent} that the latter relation has a useful 
descent consequence, if $\rho$ is furthermore absolutely irreducible:
up to twisting by a character of $G_K$, the representation $\rho$ 
descends to one over $\Q$. The parity of this descended representation 
is intrinsic to $\rho$ and can be described in terms of its projective 
splitting field (see \cref{prop:cm}).

We also obtain strong restrictions on the image of an irreducible
$\rho$ under an additional splitting hypothesis. Suppose that the
characteristic polynomial of every element of $\rho(G_K)$ splits over
$\Fell$. Then under mild conditions, we show in \cref{prop:inducedorsixchars} that the projective image is dihedral.  It
follows that $\rho$ is induced from a quadratic extension $K'/K$, and moreover $K'$ is biquadratic over $\Q$.
This gives an effective test for ruling out the irreducible case.

We apply these methods to the Picard curves 
identified in \cite{CG26} as possibly having a
reducible $\lam$-torsion representation for a prime $\lam$ of $K$ above a rational prime $\ell \equiv 1 \pmod 3$. 
This includes $381$ curves for $\ell=7$ and $5$ curves for $\ell=13$.
For every curve considered, the Euler factors over $K$ in our computed range split 
completely modulo $7$, providing strong evidence for the splitting 
hypothesis. Our main result is:

\begin{thm}[\cref{thm:ell7_complete_decomposition}+\cref{thm:ell13_complete_decomposition}]
\label{thm:main_result}
Let $C$ be one of the Picard curves mentioned above with $\ell=7$ or $13$ accordingly. 
Suppose that the $\lam$-torsion representation $\rho_\lam$ of the Jacobian of $C$ is indeed reducible for a prime $\lam$ of $K$ above $\ell$, and the splitting modulo $\ell$ of Euler factors over $K$ persists at every prime of good reduction.
Then $\rhoellKss$ is a direct sum of six characters.
\end{thm}

For two of the $381$ curves with $\ell=7$, we also verified the result directly in \cref{ex:C3C4_fullydecomposedrepresentation} by exhibiting two independent $7$-torsion points over $K$. This is done by exhibiting a $\Q$-rational $7$-torsion point, and acting on it by the order-$3$ automorphism.

Although the computations ultimately point toward complete
reducibility rather than toward new irreducible representations,
the methods developed here also let us computationally access and study an irreducible $2$-dimensional subquotient $\rho$. In particular, we give effective procedures for
identifying the blocks of $1$-dimensional constituents of $\rhoellKss$ (\cref{alg:find_varepsilon}),
the determinant character of $\rho$ (\cref{alg:find_det}),
and for computing Frobenius traces under $\rho$ (\cref{alg:find_trace_1,alg:find_trace_2}).
We hope that these algorithms will be useful for future works, especially for finding the cuspidal Bianchi eigenforms predicted by Serre's modularity conjecture.

The paper is organized as follows.
In \cref{sec:prelim} we fix notation and conventions.
In \cref{sec:reps} we study the representations arising from Picard
curves and the constraints imposed by complex conjugation, including
the descent result for an irreducible $2$-dimensional constituent.
In \cref{sec:tracetest} we use Dickson's classification
of subgroups of $\GL_2(\Fell)$ to study the $2$-dimensional
constituent closely, in particular obtaining a criterion for 
ruling out this irreducible case.
In \cref{sec:algorithms} we give algorithms for identifying
one-dimensional constituents and for recovering the
determinant and Frobenius traces of the complementary
$2$-dimensional subquotient. Finally, in \cref{sec:results} we
apply these algorithms to the possibly reducible $\ell$-torsion 
datasets for $\ell=7$ and $13$,
proving \cref{thm:ell7_complete_decomposition} and
\cref{thm:ell13_complete_decomposition}.
\cref{sec:further_questions} discusses some further questions.

Magma \cite{Magma} implementations of our algorithms, and collected 
data are available in the GitHub repository \cite{ReducibleRepsPicardGithub}.

\subsection{Acknowledgments}
 
We thank the University of Wisconsin--Madison and the Madison
Experimental Mathematics Lab.
We thank Ivan Aidun and Eshaan Bhansali for their help in coordination.
We also thank John Cremona and Tobias Berger for helpful discussions.

SC was partially supported by NSF grant DMS-2301386.

\section{Notation and conventions}
\label{sec:prelim}

For a number field $K$, we write $G_K=\Gal(\overline K/K)$ for 
its absolute Galois group. All Galois representations, i.e., 
representations of $G_K$ are understood to be continuous. If 
$\pp$ is a finite prime of $K$ at which a 
representation is unramified, we write $\Frob_\pp$ for an 
arithmetic Frobenius element. If $K/\Q$ is Galois, let 
$c \in \Gal(K/\Q)$ denote a nontrivial element, as well as 
a fixed lift of it to $G_\Q$ when no confusion can arise.
If $\rho:G_K\to\GL_n(\F)$ is a representation, the Galois 
conjugate $\rho^c$ of $\rho$ by $c$ is defined by
\[
    \rho^c(g)=\rho(cgc^{-1}),\qquad g\in G_K.
\]
Its isomorphism class is independent of the chosen lift of $c$.
For a character $\varepsilon:G_K\to\F^\times$, the character
$\varepsilon^c$ itself is independent of this choice. If
$\pp$ and $\ppbar=c(\pp)$ are two primes of $K$ above 
a rational prime $p$, and the character $\varepsilon$ is 
unramified at both $\pp$ and $\ppbar$, then $c \Frob_\pp c^{-1}$ is conjugate 
to $\Frob_{\ppbar}$, so we have:
\[
    \varepsilon^c(\Frob_\pp)
      =\varepsilon(\Frob_{\ppbar}).
\]
We freely identify finite-order characters of $G_K$ with the
corresponding finite-order Hecke characters of $K$, or 
equivalently ray class characters, via global class field 
theory.

Throughout this article, $K=\Q(\zeta_3)$, and $c$ denotes the 
nontrivial element of $\Gal(K/\Q)$, as well as a fixed lift of 
it to $G_\Q$ representing complex conjugation. We refer to the 
Galois conjugate by $c$ simply as the Galois conjugate 
representation.

\section{Reducible representations of Picard Curves}
\label{sec:reps}
 
Let $C/\Q$ be a genus-$3$ curve of the form $y^3 = f_4(x)$ for a
separable quartic polynomial $f_4$. The curve has an order-$3$
automorphism $[\zeta_3]: (x,y) \mapsto (x,\zeta_3 y)$ over
$K=\Q(\zeta_3)$, so that the geometric endomorphism ring of its Jacobian
$A$ contains $\Z[\zeta_3]$. Let $\ell \equiv 1 \pmod 3$ be a rational
prime $\ell$. Since $\ell$ splits in $\OK=\Z[\zeta_3]$ as $\lambda
\lambar$, the $\ell$-torsion subgroup $A[\ell]$ decomposes as a direct
sum $A[\lambda] \oplus A[\lambar]$, where the components $A[\lambda]$
and $A[\lambar]$ are both maximal isotropic subgroups. Hence, over $K$,
the associated Galois representation $\rhoellK := \rho_{A,\ell}|_{G_K} \colon
G_K \to \GSp(6,\Fell)$ coming from the Galois action on $A[\ell]$, takes
the form
\begin{align}
\label{eq:rhoell_decomp}
    \begin{pNiceArray}{cc}
    \rho_\lam & \mathbf{0} \\
    \mathbf{0} & \chi_\ell \rho_\lam^{-T}
    \end{pNiceArray},
\end{align}
where $\rho_\lam : G_K \rightarrow \GL_3(\Fell)$ is the $3$-dimensional
Galois representation on $\lam$-torsion $A[\lam]$, $\chi_\ell$ is the
mod-$\ell$ cyclotomic character taking values in $\Fell^\times$, and
$(\cdot)^{-T}$ denotes the inverse transpose. The fact that
$\rho_{\lambar}$ is isomorphic to $\chi_\ell \rho_\lam^{-T}$ comes from
identifying $A[\lambar]$ with the dual of $A[\lam]$ through the Weil
pairing, and using Galois equivariance of the Weil pairing.

We will consider primes $\ell$ for which the $3$-dimensional 
representation $\rho_\lam$ is reducible.
By switching the role of $\lam$ and $\lambar$, if necessary, we may
assume that $\rho_\lam$ has a $1$-dimensional subrepresentation. That
is, $\rho_\lam$ has the shape:
\begin{align}
\label{eq:reducible_rholam}
\rho_\lam = \begin{pNiceArray}{c|cc}
\varepsilon & * & * \\
\hline
0 & \Block{2-2}<\Large>{\rho} \\
0 & \\
\end{pNiceArray},
\end{align}
where $\varepsilon$ is a character $G_K \to \Fell^\times$ and $\rho
\colon G_K \to \GL_2(\Fell)$ is the complementary $2$-dimensional 
quotient representation. In this situation, we have that the
semisimplification of $\rho_\ell|_{G_K}$ is $\rhoellKss =
\varepsilon \oplus \rho^{ss} \oplus (\chi_\ell \otimes \rho^{-T})^{ss} \oplus
(\chi_\ell \otimes \varepsilon^{-1})$.

We first study some general properties of the $2$-dimensional
representation $\rho$, assuming it is irreducible. We start with a basic
lemma relating $2$-dimensional representations and their duals.
\begin{lem}
  \label{lem:gl2dual}
  Let $\rho$ be any $\GL_2$-valued representation. Then $\rho^{-T} \cong
  (\det\rho)^{-1} \otimes \rho$.
\end{lem}
\begin{proof}
The wedge pairing on a rank-$2$ module $V$ is alternating and perfect,
so $V \cong \Hom(V, \textstyle\bigwedge^2 V) = V^\vee \otimes \det V$
equivariantly.
Concretely, with $J = \left(\begin{smallmatrix}
0&1\\-1&0\end{smallmatrix} \right)$
one can check that for all $A \in \GL_2$,
\begin{align*}
  A^T J A = \det(A) J, \quad \text{i.e., } A^{-T} = \det(A)^{-1} J A
  J^{-1}
\end{align*}
So the matrix $J$ intertwines $\rho^{-T}$ with $(\det\rho)^{-1}\rho$.
\end{proof}

We fix an element $c \in G_\Q$ representing complex conjugation. Then
$c$ restricts to the non-trivial automorphism in $\Gal(K/\Q)$. Moreover
$\rho_\ell(c)$ interchanges the two $3$-dimensional
$G_K$-representations $\rho_\lam$ and $\rho_{\lambar}$. This places 
severe restrictions on $\rho$ and $\rho_\ell$.

\begin{prop}
  \label{prop:detrho_conjugatedual}
  Let $C\colon y^3 = f_4(x)$ be a Picard curve over $\Q$ and $\ell
  \equiv 1 \pmod{3}$ a prime which splits as $\lam\lambar$ in $K =
  \Q(\zeta_3)$.
  Suppose that the $3$-dimensional subrepresentation $\rho_\lambda$ of
  $\rho_\ell|_{G_K}$ is reducible, admitting a $1$-dimensional 
  subrepresentation given by the character $\varepsilon:
  G_K \rightarrow \Fell^\times$ such that
  the corresponding $2$-dimensional quotient $\rho$ is irreducible. Then
  \begin{align}
    \label{eq:varepsilon_and_det}
    \varepsilon \varepsilon^c = \chi_\ell \qquad \text{and } \qquad \det(\rho) \det(\rho)^c = \chi_\ell^2
  \end{align}
  and furthermore, we have the direct sum decompositions:
  \begin{align}
    \label{eq:full_decomposition}
      \rho_\lam \simeq \varepsilon \oplus \rho \qquad \text{and } \qquad
      \rho_{\lambar} = \varepsilon^c \oplus \rho^c \simeq \chi_\ell \varepsilon^{-1} \oplus \chi_\ell \otimes \rho^{-T}.
  \end{align}
\end{prop}
\begin{proof}
  Since $\rho_\ell$ is a $G_\Q$-representation, the element 
  $\rho_\ell(c)$ interchanges the two $3$-dimensional 
  $G_K$-representations $\rho_\lam$ and $\rho_{\lambar}$, giving an 
  isomorphism between $\rho_{\lambar}$ and the Galois conjugate
  $\rho_\lam^c$ of $\rho_\lam$:
  \begin{align}
  \label{eq:isom_conjrep_dualrep}
      \rho_\lam^c \simeq \rho_{\lambar} \simeq \chi_\ell \otimes
      \rho_\lam^{-T}.
  \end{align}

  Since $\rho$ is irreducible, we know by definition that 
  $\rho_\lam^{ss} \simeq \varepsilon \oplus \rho$. 
  By taking semisimplifications of \cref{eq:isom_conjrep_dualrep}, we 
  get that $\varepsilon^c \oplus \rho^c
  \simeq \chi_\ell \varepsilon^{-1} \oplus \chi_\ell \otimes \rho^{-T}$.
  Since $\rho$ is irreducible, this means that 
  $\varepsilon^c \simeq \chi_\ell \varepsilon^{-1}$ and 
  $\rho^c \simeq \chi_\ell \otimes \rho^{-T}$. Taking determinants leads 
  to $\det(\rho)^c = \det(\rho^c) = \chi_\ell^2 \det(\rho^{-T}) =
  \chi_\ell^2 \det(\rho)^{-1}$.
  
  Now the fact that $\varepsilon$ is a $1$-dimensional
  \emph{sub}representation of $\rho_\lam$, implies that $\varepsilon^c$
  is also a $1$-dimensional \emph{sub}representation of
  $\rho_{\lambar}$. Let us call this space $U$. Since $\rho_{\lambar}
  \simeq \chi_\ell \otimes (\rho_\lam)^{-T}$ is the symplectic dual with
  respect to the Weil pairing, the $2$-dimensional representation
  corresponding to the dual $\chi_\ell \otimes \rho^{-T}$ of $\rho$ is
  again a \emph{sub}representation of $\rho_{\lambar}$. Let us call this
  space $W$. Since $\rho$ is assumed to be irreducible, we deduce that
  $W$ is also irreducible, and thus $W \cap U = 0$. This shows that
  $\rho_{\lambar} \simeq \rho_{\lambar}^{ss}$. By conjugation we deduce
  the same for $\rho_\lam$ and hence $\rho_\lam \simeq \varepsilon
  \oplus \rho$. This proves the decompositions of $\rho_\lam$ and 
  $\rho_{\lambar}$.
\end{proof}

\begin{cor}
\label{cor:completereducibility}
If $\varepsilon \varepsilon^c \ne \chi_\ell$,
then $\rhoellKss$ decomposes completely as a direct sum of
six characters. In particular, $\rhoellK$ contains no
irreducible $2$-dimensional subquotient.
\end{cor}
\begin{proof}
If $\varepsilon\varepsilon^c \ne \chi_\ell$, then $\rho$ cannot be 
irreducible, by \cref{prop:detrho_conjugatedual}. This means that 
$\rho_\lam^{ss}$ decomposes completely into a direct sum of $3$ 
characters. Since $\rhoellKss = \rho_\lam^{ss} 
\oplus \rho_{\lambar}^{ss} \simeq \rho_\lam^{ss} 
\oplus (\rho_\lam^{ss})^c$ by \cref{eq:rhoell_decomp}, we deduce that
$\rhoellKss$ is a direct sum of six characters.
\end{proof}

Using \cref{lem:gl2dual}, the isomorphism 
$\rho^c \simeq \chi_\ell \otimes \rho^{-T}$ appearing in the proof of 
\cref{prop:detrho_conjugatedual} becomes a twist condition.
\begin{equation}
  \label{eq:hyp}
  \rho^c \cong \chi_\ell \otimes \rho^{-T} \cong \psi \otimes \rho,
  \qquad
  \text{where } \psi := \chi_\ell(\det\rho)^{-1}.
\end{equation}
In particular, the projective representation $\PP\rho$ is invariant 
under Galois conjugation, i.e., $(\PP\rho)^c \cong
\PP\rho$.

\begin{prop}
\label{prop:descent}
  Let $\rho: G_K \rightarrow \GL_2(\F_\ell)$ be absolutely irreducible
  and satisfy the twist relation in \eqref{eq:hyp}. Then there exists 
  a continuous representation $\tilde\rho : G_\Q \to \GL_2(\Fbar_\ell)$ 
  and a finite-order character $\eta$ of $G_K$ such that 
  $\rho \cong \eta \otimes \tilde\rho|_{G_K}$.
  Moreover, the projective representation $\PP(\tilde\rho)$ produced 
  by the above twist relation, and its parity, are independent of all 
  choices.
\end{prop}
\begin{proof}
Since $\rho^c \cong \psi \otimes \rho$, choose an intertwiner $A$, i.e.,
$A \in \GL_2(\F_\ell)$ with $\rho^c(g) \coloneqq \rho(c g c^{-1}) = A
\psi(g)\rho(g)A^{-1}$ for all $g \in G_K$. Applying this twice and
noting that
$\psi^c(g) \psi(g) = 1$ by \cref{eq:varepsilon_and_det}, we see that
\begin{align*}
  \rho(g) = \rho(c^2) \rho(g) \rho(c^{-2}) = A^2 \psi^c(g) \psi(g)
  \rho(g) A^{-2} = A^2 \rho(g) A^{-2} \quad \text{for all } g \in G_K.
\end{align*}
Since $\rho$ is absolutely irreducible, Schur's lemma implies that $A^2
\in \Fbar_\ell^\times$. Since $A \in \GL_2(\F_\ell)$, we get that $A^2 =
\mu$ for some $\mu \in \F_\ell^\times$. Then the projective
representation $\PP\rho$ extends to a projective representation of
$G_\Q$ by mapping $c$ to $[A] \in \PGL_2(\F_\ell)$. Since $H^2(G_\Q,
\Q/\Z) = 0$ \cite[Part II \S 6.4]{Ser77}, and $\Fbar_\ell^\times$ is the
prime-to-$\ell$ part of $\Q/\Z$, we get $H^2(G_\Q, \Fbar_\ell^\times) =
0$. So the projective $G_\Q$-representation lifts to a continuous linear
representation $\tilde\rho \colon G_\Q \to \GL_2(\Fbar_\ell)$. Finally
since $\PP(\tilde\rho|_{G_K}) = \PP\rho$, the two representations
$\tilde\rho|_{G_K}$ and $\rho$ differ by a character $\eta$. It has
finite order because it is continuous, being a ratio of continuous
representations.

Now we prove the last statement about $\PP(\tilde\rho)$ and
$\det(\tilde\rho(c))$ being uniquely determined by the twist relation. 
We need to
show that different choices of complex conjugation or the intertwiner do
not change them. A different choice of complex conjugation amounts to
replacing $c$ by $h c h^{-1}$ for some $h \in G_K$, and it is easy to
check that this changes $A$ to $\rho(h) A \rho(h)^{-1}$, thus leaving
$\PP(\tilde\rho)$
unchanged. By Schur's lemma, a different choice of the intertwiner $A$
amounts to multiplying $A$ by a scalar matrix, which
leaves the projective $G_\Q$-representation $\PP(\tilde\rho)$ unchanged.
Thus
$\PP(\tilde\rho)$
is uniquely determined by $\rho$.

Any two lifts of $\PP(\tilde\rho)$ differ by a character $\nu$ of
$G_\Q$, so $\det(\tilde\rho \nu)(c) = \det(\tilde\rho(c)) \nu(c)^2 =
\det(\tilde\rho(c))$. Therefore the parity of $\tilde\rho$ is also
uniquely determined by $\rho$.
\end{proof}

\begin{remark}
  \label{rem:UniqueProjectiveExtension}
  \cref{prop:descent} says that the projective representation 
  $\PP \tilde\rho$ is canonically determined by the isomorphism 
  $\rho^c \simeq \psi \otimes \rho$ of \cref{eq:hyp}. It does not say 
  that $\PP \tilde\rho$ is the unique extension of $\PP \rho$ to 
  $G_\Q$, since a priori $\rho$ might have non-trivial self-twists.
\end{remark}

The parity of the lifted representation $\tilde\rho \colon G_\Q \to
\GL_2(\Fbar_\ell)$ in \cref{prop:descent} can be explicitly given in
terms of the intertwiner $A$ chosen
in the proof of \cref{prop:descent}, even though it does not depend on
this choice. Since $\PP\tilde\rho(c) = [A]$, we can write $\tilde\rho(c)
= \alpha A$ for some $\alpha \in \Fbar_\ell^\times$. Then we get that $I
= \tilde\rho(c)^2 = \alpha^2 A^2$, giving the formula:
\begin{align}
\label{eq:parity}
  \det\tilde\rho(c) = \alpha^2 \det A = \frac{\det A}{\mu},
\end{align}
where $\mu \in \F_\ell^\times$ is such that $A^2 = \mu I$. Let $F
\coloneqq \Qbar^{\,\ker \PP(\tilde\rho)}$ and $L \coloneqq \Kbar^{\,\ker
\PP(\rho)}$ be the splitting fields of the projective representations
$\PP \tilde\rho$ and $\PP \rho$ respectively. Then it is known that
$\det(\tilde\rho)(c) = 1 \iff \PP\tilde\rho(c) = 1 \iff F$ is totally
real. We now describe the parity of $\tilde\rho$ purely in terms of the
splitting field $L$ of $\PP\rho$, under a mild condition on $\rho$.

\begin{prop}
  \label{prop:cm}
  Let $L \coloneqq \Kbar^{\,\ker\PP\rho}$ be the projective splitting
  field of $\rho$. Then $L$ is Galois over $\Q$ and $L = KF$, where $F
  \coloneqq \Qbar^{\,\ker \PP(\tilde\rho)}$ is the projective splitting
  field of $\tilde\rho$. 
  If $\tilde\rho$ is even, then $L$ is a CM field.
  Converse holds if $\rho$ has no non-trivial self-twists.
\end{prop}

\begin{proof}
By \cref{prop:descent}, we know that $\PP\tilde\rho|_{G_K} \cong
\PP\rho$. Hence $\ker \PP\rho = G_K \cap \ker \PP\tilde\rho$. Since both
$G_K$ and $\ker \PP\tilde\rho$ are kernels of homomorphisms from $G_\Q$,
they are normal in $G_\Q$, which means that $\ker \PP\rho$ is normal in
$G_\Q$. Thus $L/\Q$ is Galois. Furthermore, taking fixed fields also
gives $L = KF$.

Suppose $\tilde\rho$ is even. Then $F$ being the splitting field of $\PP
\tilde\rho$ is totally real. Since $K$ is an imaginary quadratic field,
and $L = KF$, we immediately get that $L$ is a CM field with $L^+ = F$.

Conversely, suppose $L$ is CM with maximal totally real subfield $L^+$.
Since $K \subseteq L$ is imaginary we must have $L = KL^+$, and
$\Gal(L/\Q) \cong \Gal(L/K) \times \Gal(L/L^+)$. Note that $\Gal(L/K)$
is isomorphic to the image of $\PP \rho$ in $\PGL(2,\F_\ell)$ and
$\Gal(L/L^+) = \langle c \rangle$, where $c$ represents complex
conjugation on $L$. Therefore $\Pi(\sigma, c^i) \coloneqq
\PP\rho(\sigma)$ defines a homomorphism $\Gal(L/\Q) \to \PGL_2(\F_\ell)$
extending $\PP\rho$. By our assumption about self-twists of $\rho$, 
and the uniqueness shown in \cref{prop:descent}, we deduce that 
$\Pi = \PP\tilde\rho$. Therefore $\PP\tilde\rho(c) = \Pi(1,c) = 1$, 
and hence $\tilde\rho$ is even.
\end{proof}

We now show that the parity of $\tilde\rho$ is unconstrained by the
twist condition in \cref{eq:hyp}, at least as long as we allow ourselves to consider representations $\rho$ with coefficients in $\Fellbar$. We start with a lemma about the norm
map on the group of characters of $G_K$. For any field $F$, let $X_F$
denote the group of continuous characters $G_F \to \Fellbar^\times$.

\begin{lem}
  \label{lem:NormOfCharacterGroup}
  The image of the natural map $X_\Q \to X_K$ obtained by restriction is
  equal to the subgroup $X_K^c$ of conjugation-invariant characters of
  $X_K$. Furthermore, if $\Norm \colon X_K \to X_K$ is the norm map that
  sends a character $\eta : G_K \rightarrow \Fellbar$ to $\Norm(\eta)
  \coloneqq \eta \eta^c$, then $\Norm(X_K) = X_K^c$.
\end{lem}
\begin{proof}
Let $c \in G_\Q$ denote a fixed element representing complex
conjugation. We will denote the restriction of this element to $\Gab_\Q$
and $\Gal(K/\Q)$ also by $c$.
Consider the inflation--restriction sequence for the trivial
$G_\Q$-module $\Fellbar^\times$, and the normal subgroup $G_K
\triangleleft G_\Q$:
\begin{align}
  \label{eq:InflationRestrictionXQtoXK}
  0 \to H^1\big(\Gal(K/\Q), \Fellbar^\times\big)
    \to X_\Q
    \xrightarrow{\ \res \ } X_K^{\Gal(K/\Q)}
    \to H^2\big(\Gal(K/\Q), \Fellbar^\times\big),
\end{align}
where $X_K^{\Gal(K/\Q)} = \{\mu \in X_K : \mu^c = \mu\} = X_K^c$. We
know that $H^2(\Z/2, \Fellbar^\times) =
\Fellbar^\times/(\Fellbar^\times)^2 = 0$ and $H^1(\Z/2, \Fellbar^\times)
= \Hom(\Z/2, \Fellbar^\times) = \langle \epsilon_K \rangle,$ where
$\epsilon_K$ is the quadratic character of $G_\Q$ cutting out $K$.
Therefore \cref{eq:InflationRestrictionXQtoXK} becomes the short exact
sequence
\begin{align}
  \label{eq:KernelOfXQtoXK}
  0 \to \Hom(\Z/2, \Fellbar^\times) \to X_\Q \to X_K^c \to 0,
\end{align}
which shows the first part of the statement.

The inclusion $\Norm(X_K) \subseteq X_K^c$ is clear since $\Norm(\eta)^c
= (\eta\eta^c)^c = \eta^c\eta = \Norm(\eta)$. For proving the reverse
inclusion, we need to first consider the Verlagerung map $\Gab_\Q \to
\Gab_K$. For a number field $F$, let $C_F=\A_F^\times/F^\times$ denote
its idèle class group and $C_F^{\circ}$ its identity component. By
global class field theory, we have a commutative diagram
\begin{equation*}
\begin{tikzcd}
C_\Q/C_\Q^{\circ}
  \arrow[r, "\Art_\Q", "\sim"']
  \arrow[d, "\iota"']
&
G_\Q^\ab
  \arrow[d, "\Ver_{K/\Q}"]
\\
C_K/C_K^{\circ}
  \arrow[r, "\Art_K", "\sim"']
&
G_K^\ab .
\end{tikzcd}
\end{equation*}
where the rows are the Artin reciprocity isomorphisms, and $\iota$ is
induced by the natural $C_\Q \hookrightarrow C_K$ (see \cite[Theorem
6.13(b)]{Neu13}). Since $K$ is imaginary quadratic, the identity
component $C_K^{\circ}$ is the image of $\C^\times$ in $C_K$. Therefore
$\ker \iota = \R^\times/\R_{>0} \cong \Z/2$, and under the Artin map
$\Art_\Q$, the non-trivial element in $\ker \iota$ maps to the class of
complex conjugation $c \in \Gab_\Q$. Therefore $\ker \Ver_{K/\Q} =
\langle c \rangle$. Consequently the image of the induced map
$\Ver_{K/\Q}^* \colon X_K \to X_\Q$ that sends $\eta \mapsto \eta \circ
\Ver_{K/\Q}$, is exactly the subgroup consisting of even characters.

Now we can prove the reverse inclusion. Let $\mu \in X_K^c$. By
\cref{eq:KernelOfXQtoXK}, it has exactly two extensions to $G_\Q$,
differing by $\epsilon_K$. Since $\epsilon_K(c)=-1$, one of them is
even. Call it $\tilde\mu$. Then $\tilde\mu \in \Ver_{K/\Q}^*(X_K)$,
i.e., $\tilde\mu = \eta\circ\Ver_{K/\Q}$ for some $\eta \in X_K$. Then
$\mu = \tilde\mu|_{G_K} = (\eta \circ \Ver_{K/\Q})|_{G_K} = \eta\eta^c =
\Norm(\eta)$, which shows that $X_K^c \subseteq \Norm(X_K)$.
\end{proof}

\begin{prop}
  \label{prop:rhotildeUnconstrainedParity}
  Consider an irreducible representation $\tilde\rho \colon G_\Q \to
  \GL_2(\Fell)$. Then there exists a character $\eta : G_K \to
  \Fellbar^\times$ such that $\rho \coloneq \eta \otimes
  \tilde\rho|_{G_K}$ satisfies the twist condition $\rho^c \cong
  \chi_\ell (\det\rho)^{-1} \otimes \rho$ of \cref{eq:hyp}.
\end{prop}
\begin{proof}
  Since $\tilde\rho$ is a $G_\Q$-representation, we know that
  $\tilde\rho|_{G_K}^c \cong \tilde\rho|_{G_K}$. Therefore the desired
  twist condition of \cref{eq:hyp} becomes equivalent to the condition
  $\Norm(\eta) = \chi_\ell (\det \tilde\rho)^{-1}|_{G_K}$. By
  \cref{lem:NormOfCharacterGroup}, since $\res(X_\Q) = X_K^c =
  \Norm(X_K)$, we deduce that there always exists a character $\eta
  \colon G_K \to \Fellbar^\times$ satisfying this condition.
\end{proof}

The results above describe the structure forced upon a $2$-dimensional 
subquotient $\rho$ of $\rho_\lam$, in the event that it is irreducible. 
In particular, if it is absolutely irreducible, 
\cref{prop:descent} shows that, up to twist, $\rho$ is the restriction 
of a $G_\Q$-representation. 
We now turn to the complementary problem of determining when 
the irreducible case can be ruled out.

\section{A criterion for ruling out an irreducible
\texorpdfstring{$2$}{2}-dimensional subquotient}
\label{sec:tracetest}

Suppose $\rho_\lambda$ admits a character $\varepsilon$ and $\rho$ is 
the $2$-dimensional quotient representation as in 
\cref{eq:reducible_rholam}. When 
$\varepsilon\varepsilon^c \ne \chi_\ell$, 
\cref{cor:completereducibility} says that $\rhoellKss$
is a direct sum of six characters, i.e., $\rho$ is reducible. 
In this section, we develop a criterion that can still allow us 
to conclude reducibility of $\rho$ even in the case that
$\varepsilon\varepsilon^c=\chi_\ell$.

Throughout $f_p$ will denote the characteristic polynomial of
$\rho_\ell(\Frob_p)$ or $\rho_\ell(\Frob_p^2)$ depending on whether $p$
is split or inert in $K = \Q(\zeta_3)$. Our standing assumption in this 
section is that for all primes $p$ of good reduction for the curve $C$, 
the characteristic polynomial $f_p$ splits completely into linear 
factors over $\Fell$. If $\rho$ is irreducible, this hypothesis 
strongly restricts its image. We denote by $G$ the image 
of $\rho$ as a subgroup of $\GL_2(\Fell)$.

We need a few lemmas. These are standard in the study of
$2$-dimensional Galois representations, and they are 
consequences of Dickson's classification of subgroups of 
$\GL_2(\Fell)$. See \cite[\S 2]{Ser72} for a review of this 
classification.

\begin{lem}
\label{lem:kleinfour}
Let $\ell \equiv 3 \pmod 4$ be a prime. Suppose $G \subseteq
\gl{2}{\Fell}$ is an irreducible subgroup such that the characteristic
polynomial of every element factors completely over $\Fell$. Then the
image of $G$ in $\PGL_2(\Fell)$ contains no subgroup isomorphic to
$(\Z/2)^2$.
\end{lem}
\begin{proof}
Suppose not. Let $g,h \in G$ be such that their projections generate a
subgroup isomorphic to $(\Z/2)^2$ of $\PGL_2(\Fell)$. For any non-scalar
$x \in \gl{2}{\Fell}$, such that $x^2$ is the scalar matrix $z \cdot
\Id_2$, the eigenvalues of $x$ are the two square roots of $z$. If $x
\in G$, then by assumption the eigenvalues are in $\Fell^\times$, and so
$z \in \Fell^{\times 2}$.

Applying this to $g$ and conjugating, we may assume that $g =
\begin{pmatrix} a
& 0 \\ 0 & -a
\end{pmatrix}$ with $a \in \Fell^\times$. Since the images of $g$ and
$h$ in
$\PGL_2(\Fell)$ commute, we have $hgh^{-1} = \mu g$ with $\mu^2 = 1 \in
\Fell^\times$. If $\mu = 1$, it forces $h$ to be diagonal and thus the
images of $g$ and $h$ generate a cyclic subgroup of $\PGL_2(\Fell)$,
which is
a contradiction. So $\mu = -1$ and $h = \begin{pmatrix} 0 & b \\ c &
0\end{pmatrix}$ for some
$b,c \in \Fell^\times$. Now $h^2 = bc \cdot \Id_2$ and $(gh)^2 = -a^2bc
\cdot \Id_2$, so both $bc$ and $-bc$ are squares in $\Fell^\times$. This
forces $-1 \in \Fell^{\times 2}$, contradicting $\ell \equiv 3 \pmod 4$.
\end{proof}

\begin{lem}
\label{lem:splittorusnormaliser}
Let $\ell \geq 5$ be a prime with $\ell \equiv 3 \pmod 4$. Suppose $G
\subseteq \gl{2}{\Fell}$ is an irreducible subgroup such that the
characteristic polynomial of every element factors completely over
$\Fell$. Then $G$ is contained in the normaliser $N(T)$ of a split
Cartan subgroup $T \subset \gl{2}{\Fell}$, but not in $T$. Furthermore, 
the image of $G$ in $\PGL_2(\Fell)$ is isomorphic to the dihedral group 
of order $2n$ for some $n \geq 3$.
\end{lem}
\begin{proof}
This is immediate from Dickson's classification of the subgroups of
$\gl{2}{\Fell}$ for $\ell \geq 5$ (see \cite[\S2]{Ser72} or
\cite{Dic58}): Since $\ell \equiv 3 \pmod 4$, \cref{lem:kleinfour} rules 
out the exceptional subgroups of type $A_4, S_4, A_5$, since the 
projective image contains a subgroup isomorphic to $(\Z/2)^2$ in each 
case. The irreducibility assumption rules out Borel subgroups. The 
assumption about every characteristic polynomial factoring completely 
over $\Fell$ rules out any non-scalar element of a non-split Cartan 
subgroup from being contained in $G$.

\noindent Irreducibility of $G$ implies that $n \neq 1$, and 
\cref{lem:kleinfour} implies that $n \ne 2$, which proves the last 
statement.
\end{proof}

\begin{remark}
\label{rem:ell7mod12}
The congruence hypothesis $\ell \equiv 3 \pmod 4$ in
\cref{lem:splittorusnormaliser} cannot in general be omitted. 
Indeed, if
$\ell \equiv 1 \pmod{12}$, there exists an irreducible subgroup
$G \subseteq \SL_2(\Fell)$ with $G \cong \SL_2(\F_3)$ such that the
characteristic polynomial of every element of $G$ splits over
$\Fell$. Its projective image is $A_4$, so $G$ is not contained in the 
normalizer of a split Cartan. Thus, in this case we instead use 
\cref{lem:projorder}, whose hypothesis can be checked by 
computing traces of $\rho$. (see \cref{subsec:results_ell13} 
for an example of verification of this hypothesis for $\ell=13$).
\end{remark}

\begin{lem}
\label{lem:projorder}
Let $\ell \geq 5$ be a prime. Suppose $G \subseteq \gl{2}{\Fell}$ is an
irreducible subgroup such that the characteristic polynomial of every
element factors completely over $\Fell$, and that $G$ contains an 
element of projective order at least $6$. Then $G$ is contained in the 
normaliser $N(T)$ of a split Cartan subgroup $T$, but not in $T$.
\end{lem}
\begin{proof}
We again use Dickson's classification in the form given by Serre
\cite{Ser72}. First, if $\ell$ divides $|G|$, then \cite[Prop.~15]{Ser72} 
implies that $G$ is either contained in a Borel subgroup or contains 
$\SL_2(\Fell)$. The former contradicts irreducibility, while the latter 
is impossible since $\SL_2(\Fell)$ contains elements lying in a nonsplit 
Cartan, whose characteristic polynomials are irreducible over $\Fell$.

Thus $\ell \nmid |G|$. By Serre's classification \cite[\S2.6]{Ser72}, 
the projective image $\overline G$ is cyclic, dihedral, or isomorphic 
to $A_4$, $S_4$, or $A_5$. Since the orders of elements of $A_4$, $S_4$, 
and $A_5$ are at most $5$ and we are given that $\overline G$ has an 
element of order at least $6$, the exceptional cases are excluded. In 
the cyclic case, $G$ is contained in a Cartan subgroup by 
\cite[\S2.6(i)]{Ser72}. Since the characteristic polynomials are assumed 
to split over $\Fell$, it forces the Cartan subgroup to be split. But 
this contradicts irreducibility of $G$.

Hence $\overline G$ is dihedral and $G$ is contained in the normalizer 
$N(T)$ of a Cartan subgroup $T$, but not in $T$. We need to show that 
$T$ is split. Every element of $N(T) \setminus T$ has projective order 
$2$. Hence an element $g \in G$ with projective order at least $6$ lies 
in $T$. Since the characteristic polynomial of $g$ splits over $\Fell$ 
and $g$ is not a scalar, it has two distinct eigenvalues in $\Fell$. 
Therefore the centralizer of $g$ in $\GL_2(\Fell)$ is a split Cartan 
subgroup. Since $T$ is abelian, contains $g$, and is a Cartan subgroup, 
it must be the centralizer of $g$. So $T$ is split.
\end{proof}

Under the hypothesis of either \cref{lem:splittorusnormaliser} or 
\cref{lem:projorder}, the conclusion is that $G$ is a subgroup of the 
normalizer $N(T)$ of a split Cartan subgroup $T$, but $G$ is not 
contained in $T$ itself.
Moreover the projective image $\overline G$ of $G$ in $\PGL_2(\Fell)$ 
is dihedral of order $2n$ for some $n \geq 3$. Furthermore, 
$G \cap T$ is the unique subgroup of index $2$ in $G$ acting reducibly.
By using the quadratic extension corresponding to this subgroup, we 
give a criterion to rule out irreducibility of the $2$-dimensional 
subquotient.

\begin{prop}
\label{prop:inducedorsixchars}
Let $C\colon y^3 = f_4(x)$ be a Picard curve over $\Q$ and $S$ the set
of primes of bad reduction for $C$. Let $\ell \equiv 1 \pmod{3}$ be a
prime, so that $\ell$ splits as $\lam\lambar$ in $K = \Q(\zeta_3)$.
Suppose that $\rho_\lam$ is reducible as in \cref{eq:reducible_rholam},
with character $\varepsilon$ and corresponding $2$-dimensional
quotient $\rho$. Suppose that either $\ell \equiv 3 \pmod{4}$ or 
the image of $G \coloneqq \rho(G_K)$ in $\PGL_2(\Fell)$ contains 
an element of order at least $6$. 
Suppose also that $f_p$ factors completely over $\Fell$ 
for every good prime $p \not\in S \cup \{3,\ell\}$. Then one of the 
following holds:
\begin{enumerate}
  \item \label{item:sixchars} $\rho_\ell|_{G_K}^{ss}$ is a direct sum of
  six characters, or
  \item \label{item:induced} there exists a squarefree integer $d \ne
  1$, with prime factors in $S \cup \{3,\ell\}$, such that $K' =
  K(\sqrt{d}) \ne K$ and
  \[
    \rho \simeq \Ind_{K'}^K \psi \qquad \text{for some character } \psi
    : G_{K'}
    \to \Fell^\times.
  \]
\end{enumerate}
In case \labelcref{item:induced}, we have $\varepsilon\varepsilon^c =
\chi_\ell$, and for every prime $\pp|p$ of $K$, which is inert in $K'/K$
with $p \not\in S \cup \{3,\ell\}$, the quotient $f_p/(T -
\varepsilon(\pp))(T - \varepsilon^c(\pp))$ is an even polynomial.
\end{prop}
\begin{proof}
If $\rho$ is reducible then each of the four summands of $\rho_\ell^{ss}
= \varepsilon \oplus \rho^{ss} \oplus (\chi_\ell \otimes \rho^{-T})^{ss} \oplus
(\chi_\ell \otimes \varepsilon^{-1})$ is a sum of characters, and we are
in case \labelcref{item:sixchars}.

Suppose that $\rho$ is irreducible. Then we have $\varepsilon 
\varepsilon^c = \chi_\ell$ by \cref{eq:varepsilon_and_det}. By Chebotarev 
density theorem, the conjugacy class of any $g \in G$ is the class of 
$\rho(\Frob_\pp)$ for a positive proportion of primes $\pp$ of $K$.
Since the primes of $K$ of residue degree $1$ has density $1$, 
we may moreover take $\pp$ to have residue degree $1$. 
For such a prime $\pp$ above $p$, one has 
$\rhoell(\Frob_\pp) = \rhoell(\Frob_p)$, so that $f_p$ is the 
characteristic polynomial of $\rho_\ell(\Frob_\pp)$. As $\rho$ is a
subquotient of $\rho_\ell|_{G_K}$, the characteristic polynomial of $g$
divides $f_p$ and hence factors completely over $\Fell$. Therefore
\cref{lem:splittorusnormaliser,lem:projorder} imply that there exists a 
split Cartan subgroup $T$ of $\GL_2(\Fell)$ with $G \subseteq N(T)$, $G
\not\subseteq T$. Let $K'|K$ be the quadratic extension cut out by the
composite $G_K \xrightarrow{\rho} G \twoheadrightarrow G/(G\cap T) \cong 
\overline G/(\overline G \cap \overline T) \cong \Z/2$. 
Since the projective image $\overline G$ is dihedral of order $2n$ with 
$n\geq3$, the subgroup $\overline G\cap\overline T$ is the unique 
cyclic subgroup of index $2$ in $\overline G$.
Since $\rho(G_{K'}) \subseteq T$, we deduce that $\rho|_{G_{K'}}$
is a direct sum of two characters. Let $\psi : G_{K'} \to \Fell^\times$
denote one of the characters. Then by Frobenius reciprocity, we get
$\rho \simeq \Ind_{K'}^K \psi$. Since $\rho_\ell$ is unramified outside
$S \cup \{\ell\}$, the same is true for the extension $K'/K$.

We claim $K'/\Q$ is Galois. On the one hand, by \cref{eq:hyp} we know 
that $\rho^c \simeq \chi_\ell \det(\rho)^{-1} \otimes \Ind_{K'}^K \psi$ 
is also induced from $K'$. On the other hand $\rho^c$ is induced from 
$(K')^c$. 
Since $\rho^c$ is projectively isomorphic to $\rho$, 
Galois conjugation preserves the unique cyclic subgroup of index $2$ 
in $\overline G$. It follows that $(K')^c=K'$. Thus $K'/\Q$ is Galois 
of degree $4$ containing $K$. Since the
quadratic subfield $K = \Q(\zeta_3)$ is imaginary, $\Gal(K'|\Q)
\not\cong \Z/4$. Therefore $\Gal(K'/\Q) \cong \Z/2 \times \Z/2$. Hence
$K' = K(\sqrt d)$ for a squarefree integer $d \neq 1$. Every prime
dividing $d$ ramifies in $K'/\Q$ and so lies in $S \cup \{3,\ell\}$.

Finally, let $\pp|p$ be a prime of $K$ which is inert in $K'/K$ with $p
\not\in S \cup \{3,\ell\}$. Let $\pp'$ denote the prime of $K'$ above
$\pp$. Then $\Frob_\pp \not\in G_{K'}$, so $\rho(\Frob_\pp) \in G
\setminus T$, and hence has characteristic polynomial $T^2 -
\psi(\Frob_{\pp'})$. Therefore $f_p = (T - \varepsilon(\pp))(T -
\varepsilon^c(\pp)) \cdot \left( T^2 - \psi(\Frob_{\pp'}) \right)\left(
T^2 - \psi^c(\Frob_{\pp'}) \right)$, which gives the last claim.
\end{proof}

\section{Algorithms}
\label{sec:algorithms}

A key ingredient for all algorithms presented in this section is the set
of Euler factors of a Picard curve $C/\Q$ at many good primes. 
These can be computed in average polynomial time, by combining the work 
\cite{Sut20} of Sutherland on computing the Euler factor mod $p$ for 
superelliptic curves, and the work \cite{AFP22} of Asif, Fite and 
Pentland to lift this to the true integral Euler factor $L_p(C,x)$ 
when $C$ is a Picard curve.

\subsection{Computing the characters of $K$ possibly occuring 
in the semisimplification of the $\ell$-torsion representation}
\label{subsec:varepsilon}

Let $S$ denote the set of primes of bad reduction for a curve $C$. Given 
the Euler factors at sufficiently many good primes for the Jacobian of 
$C$, there is a general algorithm to compute the characters
in the $\ell$-torsion Galois representation of $\Jac(C)$: find the
subgroup of all Hecke characters $G_K \rightarrow \Fell^\times$ that are
unramified outside $S$; prune this space successively by removing any 
character whose value at a prime $\pp$ of $K$ is not a root of the Euler 
factor at $\pp$ modulo $\ell$.

In the case of a Picard curve \emph{over $\Q$}, if $\lam,\lambar$ are
the primes in $K$ above a split prime $\ell$, the two $3$-dimensional
representations $\rho_\lam$ and $\rho_{\lambar}$ are not just related as
duals, but also as Galois conjugates. Using this extra structure, we can
get rid of some spurious characters, and compute the possibilities for
the set of all characters appearing in $\rho_\ell|_{G_K}^{ss}$. This
leads to the improved \cref{alg:find_varepsilon}.

\begin{prop}
  \label{prop:irredblocksofchars}
  Let $L$ be the multiset of all $1$-dimensional Jordan--H\"older
  constituents of $\rho_\ell|_{G_K}^{ss}$. Then $L$ is a disjoint union, 
  with multiplicities, of blocks of the following two types:
  \begin{enumerate}
    \item \label{item:conjandualpair} a pair of Galois conjugate
    characters $\varepsilon, \varepsilon^c$ such that $\varepsilon
    \varepsilon^c = \chi_\ell$, or
    \item \label{item:4tuple} a $4$-tuple of characters $\varepsilon,
    \varepsilon^c, \chi_\ell \otimes \varepsilon^{-1}, \chi_\ell \otimes
    (\varepsilon^c)^{-1}$, if $\varepsilon \varepsilon^c \ne \chi_\ell$.
  \end{enumerate}
\end{prop}
\begin{proof}
  By \cref{eq:isom_conjrep_dualrep}, if a character
  $\varepsilon: G_K \rightarrow \Fell^\times$ appears in the
  semisimplification of $\rho_\ell|_{G_K}$, then either $\varepsilon
  \varepsilon^c = \chi_\ell$, or each of the three characters
  $\varepsilon^c, \chi_\ell \otimes \varepsilon^{-1}, \chi_\ell \otimes
  (\varepsilon^c)^{-1}$ also appears in the semisimplification of
  $\rho_\ell|_{G_K}$.
\end{proof}

\begin{algorithm}
\label{alg:find_varepsilon}
\textbf{Input:}
\begin{itemize}
    \item A Picard curve $C\colon y^3 = f_4(x)$, and the finite set $S$
    of primes of bad reduction for $C$.
    \item a prime $\ell \equiv 1 \pmod 3$ such that $\rho_\lam$ is
    possibly reducible as in \cref{eq:reducible_rholam} for some prime
    ideal $\lam$ above $\ell$.
\end{itemize}
\textbf{Output:} A set $M$ containing multisets of Hecke characters $G_K
\rightarrow \Fell^\times$, such that one of the multisets is exactly the
set of all $\Fell^\times$-valued characters appearing in the
semisimplification $\rho_\ell|_{G_K}^{ss}$.
\begin{enumerate}
\item \label{step:HeckeCharacterGroup} Construct the group $H$ of all
Hecke characters $G_K \rightarrow \Q(\zeta_{\ell-1})$ of order dividing
$\ell-1$, that are unramified outside $S \cup \{\ell\}$. Choose a prime
above $\ell$ in $\Q(\zeta_{\ell-1})$, so that any character $\chi \in H$
may be viewed as valued in $\Fell^\times$, without any loss of
information.
\item \label{step:prune1} Let $L$ be the set of characters in $H$. For
all primes $p \not\in S \cup \{3,\ell\}$ up to a chosen bound, do the
following to recursively shrink the set $L$:
\begin{enumerate}[(a)]
  \item \label{step:charpolfp} Compute $f_p$, which is defined as the
  characteristic polynomial of $\rho_\ell(\Frob_p)$ or
  $\rho_\ell(\Frob_p^2)$ depending on whether $p$ is split or inert in
  $K = \Q(\zeta_3)$. If $p \equiv 1 \pmod 3$, this is done by reducing
  the Euler factor of $C$ at $p$ modulo $\ell$. If $p \equiv 2 \pmod 3$,
  $f_p$ is the polynomial whose roots are the squares of the roots of
  the Euler factor, reduced modulo $\ell$.
  \item Choose a prime ideal $\pp$ in $K$ above $p$. For each possible
  character $\chi \in L$, if $\chi(\pp) \in \Fell^\times$ is not a root
  of $f_p \pmod \ell$, then remove $\chi$ from $L$.
\end{enumerate}
\item \label{step:identifypairs} After the pruning step in
\cref{step:prune1}, a small set $L$ of characters will remain. 
Identify the dual pairs $\{\chi_1, \chi_2\}$ in $L$, i.e., pairs 
satisfying $\chi_1 \chi_2 = \chi_\ell$.
Identify the Galois conjugate pairs $\{\chi,\chi^c\}$ in $L$ exactly, 
or by checking the equality $\chi(\pp) = \chi^c(\ppbar)$ for all prime 
ideals $\pp$ above split primes $p$ up to a chosen large enough bound.
\item Divide $L$ in to blocks of $2$ or $4$ characters, as described in
\cref{prop:irredblocksofchars}.
\item \label{step:constructcombinations} Construct the set $M$ of all
possible combinations of blocks of characters taken with multiplicities,
as long as the number of characters counted with multiplicity is not
more than $6$.
\item \label{step:prune2} For each prime $p \not\in S \cup \{3\}$ and
corresponding characteristic polynomial $f_p$ computed in
\cref{step:prune1}, do the following: Choose a prime ideal $\pp$ in $K$
above $p$. For every combination $X$ in $M$, if the multiset $\{
\chi(\pp) : \chi \in X \}$ containing $|X|$ not-necessarily distinct
elements of $\Fell^\times$ is not a subset of the multiset of roots of
$f_p \pmod \ell$, then remove the combination $X$ from $M$.
\item Return the set $M$ of multisets of characters.
\end{enumerate}
\end{algorithm}

\begin{remark}
  \label{rem:Raynaud}
  Suppose that the Jacobian of $C$ has good reduction 
  at a prime $\ell$ that splits in $K$. Then 
  \cite[Thm. 3.4.3]{Raynaud} and \cite[Exposé IX]{SGA7}
  together constrain the action of inertia at a prime $\lam$ 
  above $\ell$ on the Jordan-Holder factors of the 
  $\ell$-torsion. In particular, 
  $\varepsilon: G_K \rightarrow \Fell^\times$ is a character 
  appearing in the semisimplification $\rho_\ell|_{G_K}^{ss}$,
  then one of $\varepsilon$ or $\chi_\ell \otimes 
  \varepsilon^{-1}$ is unramified at $\lam$. This can be used 
  to further prune the set of characters obtained in 
  \cref{step:prune1} in \cref{alg:find_varepsilon}.
\end{remark}

\subsection{Computing the Determinant Character}
\label{subsec:detrho}

In this section, we give algorithms to compute the determinant character
of the $2$-dimensional subquotient representation $\rho$, and the traces
$\operatorname{tr}(\rho(\Frob_\pp))$ at good primes $\pp \subset \OK$.
Our starting point is the same as earlier. Just like the one-dimensional
subrepresentation $\varepsilon$, we know that the determinant of $\rho$
is a Hecke character of $K$ of order dividing $\ell-1$, that is
unramified outside the primes of bad reduction for $C$, which gives a
finite set of possibilities for what $\det(\rho)$ could be.
After pruning using Euler factors at sufficiently many primes, several
candidate characters usually survive. We again exploit the Galois
conjugate structure of $\rho_\lam$ and $\rho_{\lambar}$ through 
\cref{eq:varepsilon_and_det} to narrow down this list further.

\begin{algorithm}
\label{alg:find_det}
\textbf{Input:}
\begin{itemize}
    \item A Picard curve $C\colon y^3 = f_4(x)$, and the finite set $S$
    of primes of bad reduction for $C$.
    \item a prime $\ell \equiv 1 \pmod 3$ such that $\rho_\lam$ is
    possibly reducible as in \cref{eq:reducible_rholam} for some prime
    ideal $\lam$ above $\ell$.
    \item a Galois conjugate pair $\varepsilon, \varepsilon^c$ of
    characters $G_K \rightarrow \Fell^\times$ such that $\varepsilon
    \varepsilon^c = \chi_\ell$, such that one of them represents the
    $1$-dimensional irreducible subrepresentation of $\rho_\lam$ in
    \cref{eq:reducible_rholam}.
\end{itemize}
\textbf{Output:} A set $L$ of candidate Hecke characters for
$\det(\rho)$.
\begin{enumerate}
\item Construct the group $L$ of all Hecke characters $G_K \rightarrow
\Fell^{\times}$, that are unramified outside $S \cup \{\ell\}$, exactly
as in \cref{step:HeckeCharacterGroup} of \cref{alg:find_varepsilon}.
\item \label{find_det:step:prune} For each prime $p \not\in S \cup
\{3,\ell\}$ up to a chosen bound, and the corresponding polynomial $f_p$
computed in \cref{step:charpolfp} of \cref{alg:find_varepsilon}, do the
following to recursively shrink the set $L$:
\begin{enumerate}[(a)]
\item \label{find_det:step:rootsoffp} Compute the multiset $R_1$ of
roots of $f_p$ over $\F_{\ell^2}$.
\item \label{find_det:step:removeepsilonvals} Fix a prime ideal $\pp
\subset \OK$ above $p$. Remove from $R_1$ the values
$\varepsilon(\Frob_\pp)$ and $\Norm(\pp)/\varepsilon(\Frob_\pp) \pmod
\ell$; call the remaining four roots $\{r_1,r_2,r_3,r_4\}$.
\item Form $R_2 = \{r_ir_j : i < j\}$, remove two copies of $\Norm(\pp)
\pmod \ell$ from $R_2$ (corresponding to the symplectic pairing).
\item Discard from $L$ every character that does not evaluate at $\pp$
to an element of $R_2$.
\end{enumerate}
\item \label{find_det:step:identifypairs} After the pruning step in
\cref{find_det:step:prune}, a small set $L$ of characters will remain.
Identify the Galois conjugate pairs $\{\chi,\chi^c\}$ in $L$ as in 
\cref{step:identifypairs} of \cref{alg:find_varepsilon}. Discard from 
$L$ any character $\chi$ such that $\chi \chi^c \ne \chi_\ell^2$.
\item Return $L$.
\end{enumerate}
\end{algorithm}

\begin{remark}
  \label{rem:twodetchars}
  The set returned by \cref{alg:find_det} is closed under conjugation by
  $\Gal(K|\Q)$, since none of the steps in the algorithm distinguish
  conjugate characters. If the $\lambda$-torsion representation
  $\rho_\lam$ of $\Jac(C)$ is indeed reducible and $\rho$ is as defined
  in \cref{eq:reducible_rholam}, then the output of \cref{alg:find_det}
  is guaranteed to contain at least two characters: $\det(\rho)$ and
  $\det(\rho)^c$. Note that the latter $\det(\rho)^c$ is the determinant
  of the $2$-dimensional subrepresentation of $\rho_{\lambar}$ dual to
  $\rho$. Hence these two characters cannot be told apart from each
  other just using Euler factors.
\end{remark}

\subsection{Computing the Trace}
\label{subsec:tracerho}
 
Once we have the character $\varepsilon$, and the determinant character
$\sigma = \det(\rho)$, we can compute the values of the trace function
$\tr(\rho)$ at $\Frob_\pp$ for various primes $\pp$ in $K$ of good
reduction for $C$. We have developed two methods for obtaining this trace. 
They are not guaranteed to succeed for \emph{all} primes $\pp$, and
\cref{prop:method1_succeeds_1overn,prop:method2_returns_atmosttwotraces}
exactly capture the failures. In particular, if $\sigma = \det(\rho)$ is
the mod-$\ell$ cyclotomic character, then both methods always succeed.
\cref{ex:C3C4_rhocharpols} is one such favourable instance.

\begin{algorithm}
\label{alg:find_trace_1}
\textbf{Input:}
\begin{itemize}
    \item A Picard curve $C\colon y^3 = f_4(x)$.
    \item a prime $\ell \equiv 1 \pmod 3$ such that $\rho_\lam$ is
    reducible as in \cref{eq:reducible_rholam} for some prime ideal
    $\lam$ above $\ell$.
    \item a character $\varepsilon: G_K \rightarrow \Fell^\times$ such
    that $\varepsilon \varepsilon^c = \chi_\ell$, representing the
    $1$-dimensional subrepresentation of $\rho_\lam$ in
    \cref{eq:reducible_rholam}.
    \item a character $\sigma: G_K \rightarrow \Fell^\times$ such that
    $\sigma \sigma^c = \chi_\ell^2$, representing $\det(\rho)$ where
    $\rho$ is the $2$-dimensional quotient representation of $\rho_\lam$
    in \cref{eq:reducible_rholam}.
    \item a prime ideal $\pp$ in $K$ above a prime $p \ne 3, \ell$ of
    good reduction for $C$.
\end{itemize}
\textbf{Output:} the value $\tr(\rho(\Frob_\pp)) \in \Fell$, or false.\\
\textbf{Method 1.}
\begin{enumerate}
\item Compute the characteristic polynomial $f_p$ as in
\cref{step:charpolfp} of \cref{alg:find_varepsilon}.
\item Let $a_5$ be the coefficient of $x^5$ in $f_p$ (i.e.\ the negative
of the sum of all roots).
\item If $\sigma(\pp) \equiv -\Norm(\pp) \pmod{\ell}$, return false.
\item Otherwise, let $A = \varepsilon(\pp)$ and $B = \sigma(\pp)$, and
return the trace $t_\pp = -(a_5 + A + \Norm(\pp)/A) \cdot (1 +
\Norm(\pp)/B)^{-1}$.
\end{enumerate}
\end{algorithm}

\begin{algorithm}
\label{alg:find_trace_2}
\textbf{Input:} Same as in \cref{alg:find_trace_1}.\\
\textbf{Output:} a list $T$ of candidate values for
$\tr(\rho(\Frob_\pp)) \in \Fell$.\\
\textbf{Method 2.}
\begin{enumerate}
\item Compute the characteristic polynomial $f_p$ as in
\cref{step:charpolfp} of \cref{alg:find_varepsilon}.
\item Form the multiset $R_1$ and eliminate roots as in
\cref{find_det:step:rootsoffp,find_det:step:removeepsilonvals} of
\cref{alg:find_det} so that the multiset $R_1$ has $4$ elements $r_1,
\ldots, r_4$.
\item For each pair $r_i,r_j \in R$ with $i < j$ satisfying $r_ir_j =
\sigma(\pp)$, record $r_i + r_j$ in a set $T$. Return $T$.
\end{enumerate}
\end{algorithm}

\begin{prop}
  \label{prop:method1_succeeds_1overn}
  With notation as in \cref{alg:find_trace_1}, fix a prime ideal $\pp
  \subset \OK$ above a good prime $p \notin S \cup \{3,\ell\}$.
  \begin{enumerate}
    \item Whenever $\sigma(\pp) \ne -\Norm(\pp)$,
    \cref{alg:find_trace_1} returns the correct value $t_\pp =
    \tr(\rho(\Frob_\pp))$.
    \item Let $n$ denote the order of the character $\psi :=
    \sigma\chi_\ell^{-1}$. If $2 \nmid n$, then \cref{alg:find_trace_1}
    always returns the correct trace. If $2|n$, then the density of
    primes $\pp$ for which \cref{alg:find_trace_1} returns false is
    $1/n$.
  \end{enumerate}
\end{prop}
\begin{proof}
  Let $\lambda_1,\lambda_2 \in \Fellbar$ be the eigenvalues of
  $\rho(\Frob_\pp)$, so that $t_\pp = \tr(\rho(\Frob_\pp)) =
  \lambda_1+\lambda_2$ and $B = \sigma(\pp) = \det(\rho(\Frob_\pp)) =
  \lambda_1\lambda_2$. By \cref{eq:rhoell_decomp,eq:reducible_rholam},
  the six roots of $f_p$ are the eigenvalues of $\varepsilon, \rho,
  \varepsilon^c, \rho^c \simeq \chi_\ell \otimes \rho^{-T}$ at
  $\Frob_\pp$, namely
  \[
    A=\varepsilon(\pp), \quad \lambda_1, \quad \lambda_2, \quad N/A =
    N/\varepsilon(\pp), \quad N/\lambda_1, \quad N/\lambda_2, 
    \quad \text{where } N = \Norm(\pp)
  \]
  Hence $a_5$, the negative of the sum of these roots, satisfies
  \[
    a_5 = -\Big(A + \frac{N}{A}\Big) - t_\pp - \frac{N}{\lambda_1} -
    \frac{N}{\lambda_2}
        = -\Big(A+\frac{N}{A}\Big) - t_\pp \Big(1+\frac{N}{B}\Big),
  \]
  and rearranging leads to the given formula for $t_\pp$ whenever $B \ne
  -N$.

  For the second statement, we first note that $\sigma(\pp) =
  -\Norm(\pp) = -\chi_\ell(\pp)$ is equivalent to $\psi(\pp) = -1$.
  Being valued in $\Fell^\times$, if the order of $\psi$ is $n$, then
  its image is the cyclic group $\mu_n \subset \Fell^\times$. If $2
  \nmid n$, then $-1 \notin \mu_n$, so the condition $\psi(\pp)=-1$
  never holds. If $2|n$ so that $-1 \in \mu_n$, then by the Chebotarev
  density theorem applied to the finite abelian extension of $K$ cut out
  by $\psi$, the set of primes $\pp$ with $\psi(\pp) = -1$ has density
  exactly $1/n$.
\end{proof}

\begin{prop}
  \label{prop:method2_returns_atmosttwotraces}
  With notation as in \cref{alg:find_trace_1,alg:find_trace_2}, fix a
  prime ideal $\pp \subset \OK$ above a good prime $p \notin S \cup
  \{3,\ell\}$, and let $t = \tr(\rho(\Frob_\pp))$. Then the set $T$
  returned by \cref{alg:find_trace_2} satisfies $\{t\} \subseteq T
  \subseteq \{t,-t\}$. Moreover if $|T| = 2$, then $\sigma(\pp) =
  -\Norm(\pp)$, and $T = \{ \tr(\rho(\Frob_\pp)),
  \tr(\rho(\Frob_{\ppbar})) \}$; in particular $p$ is split in $K$.
\end{prop}
\begin{proof}
  Taking $\lambda_1,\lambda_2 \in \Fellbar$ to be the eigenvalues of
  $\rho(\Frob_\pp)$, and letting $N = \Norm(\pp)$ just as in the proof
  of \cref{prop:method1_succeeds_1overn}, the six roots of $f_p$ are
  $\varepsilon(\pp), \lambda_1, \lambda_2, N/\varepsilon(\pp),
  N/\lambda_1, N/\lambda_2$. After removing $\varepsilon(\pp)$ and
  $N/\varepsilon(\pp)$, the multiset $R_1$ of \cref{alg:find_trace_2} is
  exactly $\{\lambda_1,\lambda_2,N/\lambda_1,N/\lambda_2\}$; label it so
  $r_1 = \lambda_1$, $r_2 = \lambda_2$, $r_3 = N/\lambda_1$, $r_4 =
  N/\lambda_2$. For each of the $6$ pairs $\{r_i,r_j\}$, we check
  whether $r_ir_j = \sigma(\pp) = \lambda_1\lambda_2$, and if so, what
  the value of $r_i + r_j$ is:
  \begin{itemize}
    \item $\{r_1,r_2\}$: the product is $\lambda_1\lambda_2 =
    \sigma(\pp)$ always, and $r_1+r_2 = t$.
    \item $\{r_i,r_j\}$ where $1$ or $2$ belongs to $\{i,j\}$: without
    loss of generality, assume $i=1$. If the product $r_1 r_j =
    \sigma(\pp) = r_1 r_2$, then $r_j = r_2$ and hence the sum $r_1 +
    r_j = r_1 + r_2 = t$.
    \item $\{r_3,r_4\}$: the product is $N^2/\lambda_1\lambda_2 =
    N^2/\sigma(\pp)$, which equals $\sigma(\pp)$ exactly when
    $\sigma(\pp)^2 = N^2$. When it does, $r_3+r_4 =
    N(\lambda_1+\lambda_2)/\sigma(\pp) = Nt/\sigma(\pp) = \pm t$,
    according to whether $\sigma(\pp) = N$ or $\sigma(\pp) = -N$.
  \end{itemize}
  The value $-t$ can only arise from the pair $\{r_3,r_4\}$, and only
  when $\sigma(\pp) = -N$. The last assertion follows from the fact that
  $\tr(\rho(\Frob_{\ppbar})) = \tr(\rho^c(\Frob_\pp)) = N/\lambda_1 +
  N/\lambda_2 = r_3 + r_4 = N t/\sigma(\pp)$.
\end{proof}

\subsection{Ruling out irreducible $2$-dimensional subquotients}
\label{subsec:ruleoutirreduciblerho}

The criterion developed in \cref{sec:tracetest} turns into a 
practical test: enumerate the finite set of quadratic extensions 
$K'=K(\sqrt{d})$ where $d$ is a squarefree integer divisible 
only by the primes in $S \cup \{3,\ell\}$, and try to eliminate 
each of them using the computed Euler factors. This is our next 
algorithm. If all of them are eliminated, then 
alternative \labelcref{item:induced} of 
\cref{prop:inducedorsixchars} cannot hold, and we can conclude 
that the $2$-dimensional quotient $\rho$ of $\rho_\lambda$ is 
in fact reducible.

\begin{algorithm}
\label{alg:tracetest}
\textbf{Input:}
\begin{itemize}
    \item A Picard curve $C\colon y^3=f_4(x)$, the finite set $S$ of
    primes of bad reduction, and a prime $\ell\equiv1\pmod3$ such that
    $\rho_\lam$ is possibly reducible as in \cref{eq:reducible_rholam}. 
    We assume that the hypotheses of \cref{prop:inducedorsixchars} are satisfied.
    \item A block of characters $\{\varepsilon,\varepsilon^c\}$ of 
    type \labelcref{item:conjandualpair} returned by
    \cref{alg:find_varepsilon}, and the polynomials $f_p$ at the good
    primes up to a chosen bound.
\end{itemize}
\textbf{Output:} Either, under the hypotheses of
\cref{prop:inducedorsixchars}, the assertion that $\rhoellKss$ is a
direct sum of six characters, or the list of quadratic extensions
$K'=K(\sqrt d)$ that the data fail to eliminate.
\begin{enumerate}
\item \label{step:candidateKprime} Let $D$ be a set of representatives
for the squarefree integers $d \ne 1$ whose prime factors lie in $S \cup
\{3,\ell\}$, modulo the relation identifying $d$ with the squarefree
part of $-3d$. 
\item \label{step:eliminate} For each $d \in D$ and each good prime $p
\equiv 1 \pmod 3$, $p \not\in S \cup \{\ell\}$, with $\legendre{d}{p} =
-1$, choose a prime $\pp$ of $K$ above $p$ and discard $d$ unless the
quotient of $f_p$ by $(T-\varepsilon(\pp))(T-\varepsilon^c(\pp))$ is an 
even polynomial.
\item If every $d \in D$ is discarded, return the assertion that 
$\rho_\ell|_{G_K}^{ss}$ is a direct sum of six characters.
\item Otherwise, return the set of quadratic extensions $K(\sqrt{d})$ for 
the surviving $d$.
\end{enumerate}
\end{algorithm}

\section{Computational Results}
\label{sec:results}
 
We apply the above algorithms to a list of $381$ Picard curves $C\colon
y^3 = f_4(x)$, for which the $7$-torsion Galois representation is
possibly reducible, i.e., possibly has the block-matrix shape as in
\cref{eq:reducible_rholam}. These curves were found in \cite{CG26}, 
and they come from Sutherland's dataset of roughly 3 million Picard 
curves of $7$-smooth conductor.

For clarity, we distinguish the following two hypotheses on a
curve in our dataset:
\begin{enumerate}[label=(\Alph*), start=18]
    \item \label[hypothesis]{hypothesis1} the representation $\rho_\lam$
    is reducible for some prime $\lam$ of $K=\Q(\zeta_3)$ above
    $\ell$.
    \item \label[hypothesis]{hypothesis2} for every prime $p$ of good
    reduction, the characteristic polynomial $f_p$ (defined in
    \cref{sec:tracetest}) splits completely over $\Fell$.
\end{enumerate}

\subsection{Results on the computation of the character $\varepsilon$}
\label{subsec:results_epsilon1}

For each of these curves $C$, we first compute candidate characters
$\varepsilon$ appearing in $\rhoellKss$ as outlined by the general
algorithm mentioned at the start of \cref{sec:algorithms}. For each 
curve, we considered Euler factors at the first $500$ good primes. For 
all $381$ curves and the good primes $p$ considered, the characteristic 
polynomial $f_p$ computed in \cref{step:prune1} of 
\cref{alg:find_varepsilon} completely factors as a product of $6$ linear 
factors over $\Fell$. 
For $376$ out of the $381$ curves, $6$ characters are found and they 
come in conjugate pairs. For $4$ of the remaining curves (given in
\cref{ex:fourcurves_selfconjchars}), $4$ characters are found two of
which are self-Galois conjugate. The last remaining curve is given in
\cref{ex:eightcharscurve}.

\begin{ex}
    \label{ex:fourcurves_selfconjchars}
    For the four Picard curves, 
    \begin{equation}
    \begin{alignedat}{2}
    \label{eq:selfconjchars_fourcurves}
      C_1:\;& y^3 = f_1(x) = 12x^4 + 8x^3 - 48x^2 + 36x - 7.\\
      C_3:\;& y^3 = f_3(x) = 648x^4 + 289x^3 + 42x^2 + 2x.\\
      C_2:\;& y^3 = f_2(x) = 625x^4 + 550x^3 + 126x^2 - 28x - 4.\\
      C_4:\;& y^3 = f_4(x) = 1458x^4 + 557x^3 + 129x^2 + 15x + 1.
    \end{alignedat}
    \end{equation}
    we obtain exactly four candidate characters that can possibly 
    appear in $\rhoellKss$. 
    In each case, exactly one of them is unramified at
    both primes above $\ell$. The conductor of this character is
    $24(1-\zeta_3)\OK$ and $60(1-\zeta_3)\OK$ for $C_1$ and $C_2$
    respectively. For $C_3$ and $C_4$, it is the trivial character.
\end{ex}

\begin{ex}
    \label{ex:eightcharscurve}
    For the curve $C_5 \colon y^3 = 27 x^4 + 36 x^3 - 396 x^2 + 504 x +
    588$, eight candidate characters are found to possibly
    appear in $\rhoellKss$.
\end{ex}

\cref{tab:characters_left} breaks down, for each curve, how many of its
candidate characters $\varepsilon$ are unramified at both primes
$\lam,\lambar$ above $\ell$, ramified at only one of the two, or
ramified at both.
Using \cite{BKSW20} we manage to successfully compute the conductor for 
$369$ out of the $381$ curves in our dataset. Among them $109$ curves 
have good reduction at $\ell=7$. While it is in principle possible that some 
spurious characters may be eliminated using \cref{rem:Raynaud} when the 
Jacobian has good reduction at $\ell$, this does not happen for any of 
these curves.

\begin{center}
\begin{table}[h!]
\begin{center}
\caption{Ramification of the candidate characters $\varepsilon$ possibly
appearing in $\rho_{\Jac(C),\ell}|_{G_K}^{ss}$,
at the two primes $\lam,\lambar$ above $\ell = 7$, after considering
Euler
factors at the first $500$ good primes. 
}
\label{tab:characters_left}
\begin{tabular}{|p{1.5cm}|p{1.5cm}|p{1.5cm}|p{1.5cm}|p{1.6cm}|p{1.5cm}|}
\hline
Unram.\ \newline at $\lam,\lambar$ & Ram.\ \newline only at $\lam$ & 
Ram.\ \newline only at $\lambar$ & Ram.\ \newline at $\lam,\lambar$ & 
Total \newline characters & Number \newline of curves\\
\hline
1 & 1 & 1 & 1 & 4 & 3\\
1 & 0 & 0 & 3 & 4 & 1\\
\hline
2 & 1 & 1 & 2 & 6 & 107\\
2 & 0 & 0 & 4 & 6 & 5\\
0 & 1 & 1 & 4 & 6 & 28\\
0 & 2 & 2 & 2 & 6 & 75\\
0 & 3 & 3 & 0 & 6 & 50\\
0 & 0 & 0 & 6 & 6 & 111\\
\hline
0 & 4 & 4 & 0 & 8 & 1\\
\hline
\multicolumn{5}{|r|}{Total} & 381\\
\hline
\end{tabular}
\end{center}
\end{table}
\end{center}

It turns out that for all but the curve $C_5$ of \cref{ex:eightcharscurve},
exploiting the Galois conjugate structure of the representations 
does \emph{not} lead to more refined results. That is, for $380$ out of
$381$ curves, \cref{alg:find_varepsilon} does not get rid of any 
character from the list of $6$ candidate characters already obtained.
It merely partitions the list of $6$ candidate characters into one block 
each of types \labelcref{item:conjandualpair} and 
\labelcref{item:4tuple}. So if $\rho_\lam$ is indeed reducible, the 
possibilities for the set of characters appearing in $\rhoellKss$ are: 
the type \labelcref{item:conjandualpair} block alone, or its union with 
the type \labelcref{item:4tuple} block. (A type \labelcref{item:4tuple} 
block cannot occur by itself because \cref{cor:completereducibility} 
will then force complete decomposition of $\rho_\ell|_{G_K}^{ss}$.) 
For the curve $C_5$ of \cref{ex:eightcharscurve}, in addition to the two 
possibilities mentioned above, we get another distinct type 
\labelcref{item:conjandualpair} block that is incompatible with
any type \labelcref{item:4tuple} block. 
Consequently, for each of these $381$ curves, we cannot rule out the 
possibility that $\rhoellKss$ decomposes completely as a direct sum of 
$6$ characters, i.e., that $\rho$ is in fact reducible.
The full list of compatible combinations of characters returned by 
\cref{alg:find_varepsilon}, for all $381$ curves, is available in the 
file \texttt{allcharsoutput.txt} in the accompanying code repository 
\cite{ReducibleRepsPicardGithub}.

For the curves $C_3, C_4$ mentioned in \cref{ex:fourcurves_selfconjchars},
we find that the type \labelcref{item:4tuple} block is 
$\{ \mathds{1}, \mathds{1}, \chi_\ell, \chi_\ell \}$. We prove it by
finding two independent $7$-torsion points on their Jacobians over 
$\Q(\zeta_3)$.
\begin{ex}
  \label{ex:C3C4_fullydecomposedrepresentation}
  For the curves $C_3$ and $C_4$ in \cref{ex:fourcurves_selfconjchars}, 
  the output of \cref{alg:find_varepsilon} suggests:
  \begin{align*}
    \rhoellKss \simeq \varepsilon \oplus \varepsilon^c \oplus 
    \mathds{1} \oplus \mathds{1} \oplus \chi_\ell \oplus \chi_\ell
    \quad \text{for some character } \varepsilon \text{ with } 
    \varepsilon\varepsilon^c = \chi_\ell.
  \end{align*}
  We compute that 
  $\Jac(C_4)(\Q)[7] \ne \varnothing$. The degree $0$
  divisor $T = [(0,1) - \infty]$, where $\infty$ is the point at
  infinity on $C_4$, is a $7$-torsion point on the Jacobian:
  since $(28x^2+7x+1)^3 - (2x+1)^3 f_4(x) = -11664x^7$, the divisor $7
  T$ is the principal divisor coming from the function $h = (2x+1)y -
  (28x^2+7x+1)$. Consequently, $[(0,\zeta_3)-\infty]$ is another
  independent $7$-torsion point, leading to $(\Z/7)^2 \subseteq
  \Jac(C_4)(\Q(\zeta_3))[7]$.

  \noindent We also find by computation that in fact
  $\Z/3 \times \Z/21 \subseteq \Jac(C_3)(\Q)$. The degree $0$ divisor
  $T = [(-1/6,-1/6) - \infty]$, where $\infty$ is the point at infinity
  on $C_3$, is a $21$-torsion point on the Jacobian. A witness
  is the function $h = A_2 y^2 + A_1 y + A_0$, where
  \begin{align*}
  A_2 &= 9(4032x^4+288x^3-543x^2-128x-8),\\
  A_1 &= -3(131328x^5+73872x^4+13023x^3+330x^2-114x-8), \\
  A_0 &= -4(6x+1)(46656x^6-75816x^5-58509x^4-14022x^3-1350x^2-36x+1).
  \end{align*} 
  It is easy to verify that the resultant $\Res_y(h, y^3-f_3(x)) =
  -64(6x+1)^{21}$ and $\divisor(h) = 21 T$. The automorphism $[\zeta_3]$
  leads to another independent $21$-torsion point over $\Q(\zeta_3)$,
  leading to $(\Z/21)^2 \subseteq \Jac(C_3)(\Q(\zeta_3))$.

  \noindent Thus $\mathds{1}$ occurs with multiplicity at least $2$ in
  $\rhoellKss$. Since $\rhoellKss$ is self-dual up to twist by
  $\chi_\ell$, the character $\chi_\ell$ also occurs with multiplicity
  at least $2$. This verifies the type \labelcref{item:4tuple} block
  $\{\mathds{1},\mathds{1},\chi_\ell,\chi_\ell\}$
  predicted by \cref{alg:find_varepsilon}.
\end{ex}

The preceding computations determine the possible
$1$-dimensional Jordan--H\"older constituents of $\rhoellKss$.
We now record the computations of the determinant
and Frobenius traces of the complementary $2$-dimensional
subquotient $\rho$. This data will be used in
\cref{subsec:results_ell13} to detect elements of large
projective order.

\subsection{Determinants and Frobenius traces of the
\texorpdfstring{$2$}{2}-dimensional subquotient}
\label{subsec:results_tracedetrho}

We recall that for all but $1$ curve in the dataset, a unique Galois
conjugate pair (type \labelcref{item:conjandualpair}) of characters
$\varepsilon, \varepsilon^c$ was found by \cref{alg:find_varepsilon}, as 
noted in \cref{subsec:results_epsilon1}. We ran \cref{alg:find_det} for
these $380$ curves with the corresponding characters, using Euler
factors at the first $500$ good primes. For $349$ curves,
\cref{alg:find_det} returns a unique pair of Galois conjugate characters
$\sigma, \sigma^c$ such that $\sigma \sigma^c = \chi_\ell^2$,
and hence determines $\det(\rho)$ up to Galois conjugation. 
We note that the naive approach only using the pruning step in 
\cref{find_det:step:prune} of \cref{alg:find_det} always leads to a set 
of $3$ or $4$ candidate characters for $\det(\rho)$. So using the Galois 
conjugate information as in \cref{find_det:step:identifypairs} is crucial
here. For the remaining $31$ curves, even this is not enough, and two 
such Galois conjugate pairs survive.

Whenever \cref{alg:find_det} succeeds in determining a unique pair
$\sigma, \sigma^c$ of conjugate characters for $\det(\rho),
\det(\rho^c)$, assuming $\sigma = \det(\rho)$, \cref{alg:find_trace_1}
finds the trace of $\rho$ at $\Frob_\pp$ for any prime $\pp$ in $K$ such
that $\sigma(\Frob_\pp) \ne -\Norm(\pp)$. Alternatively,
\cref{alg:find_trace_2} always works, and when it does not return a
unique value, the output is the pair $\{\tr(\rho(\Frob_\pp)),
\tr(\rho(\Frob_{\ppbar}))\}$, and it is necessarily of the form
$\{t,-t\}$ for some $t \in \Fell$, by
\cref{prop:method2_returns_atmosttwotraces}. 

\begin{ex}
  \label{ex:C3C4_rhocharpols}
  For the curves $C_3$ and $C_4$ from 
  \cref{ex:fourcurves_selfconjchars}, 
  \cref{alg:find_det} returns a unique self-Galois 
  conjugate character for $\det(\rho)$ and it is the $7$-torsion 
  cyclotomic character $\chi_7$. Therefore, by 
  \cref{prop:method1_succeeds_1overn}, we know that
  \cref{alg:find_trace_1} finds the trace of $\rho(\Frob_\pp)$ for any
  prime ideal $\pp$ above a prime $p \ne 3,7$ of good reduction for
  $C_3$, and furthermore $\tr(\rho(\Frob_\pp)) =
  \tr(\rho(\Frob_{\ppbar}))$. We computed the traces for the first $500$
  rational primes $p$, and in this range, the trace at $\pp$ is always 
  $1+\Norm(\pp)$, in line with what is shown in
  \cref{ex:C3C4_fullydecomposedrepresentation}.
\end{ex}

\begin{ex}
  \label{ex:C5_rhocharpols}
  Consider the curve $C_6: y^3 = 2x^4 + 4x^3 + 8x^2 + 4x + 2$.
  \cref{alg:find_det} returns a unique self-Galois conjugate character
  for $\det(\rho)$. It is the product of the quadratic character $\chi :
  G_K \rightarrow \{ \pm 1 \}$ cutting out the extension $K(i)|K$, and
  the $7$-torsion cyclotomic character $\chi_7$. Therefore, by
  \cref{prop:method2_returns_atmosttwotraces}, we know that
  \cref{alg:find_trace_2} finds the trace of $\rho(\Frob_\pp)$ exactly
  for half of all prime ideals $\pp$ above primes $p \ne 3,7$ of good
  reduction, and when it cannot find the trace exactly, it finds the set
  $\{ \tr(\rho(\Frob_\pp)), \tr(\rho(\Frob_{\ppbar})) \} = \{ t, -t \}$.
  We computed this data for the first $500$ rational primes $p$, which
  is stored in \texttt{primesdetstraces\_2ndcurve.txt}.
  \cref{tab:Trace_Polys_C5} displays this data for the first few primes.
\end{ex}

\begin{center}
\begin{table}[h!]
\begin{center}
\caption{The set of characteristic polynomials of $\rho(\Frob_\pp)$ as
$\pp$
varies over all primes in $\Q(\zeta_3)$ 
above a good prime $p \ne 3,7$, for the curve $C_6$ in
\cref{ex:C5_rhocharpols}.}
\label{tab:Trace_Polys_C5}
\begin{tabular}{|c|c||c|c|}
\hline
$p$ & The set of charpolys $\rho(\Frob_\pp)$ for $\pp|p$ & $p$ & The set
of
charpolys $\rho(\Frob_\pp)$ for $\pp|p$ \\
\hline
$11$ & $x^2+4x+2$ & $53$ & $x^2+3x+2$ \\
$13$ & $x^2+6$    & $59$ & $x^2+4x+2$ \\
$17$ & $x^2+3x+2$ & $61$ & $x^2+x+5$ \\
$19$ & $\{x^2+4x+2, x^2+3x+2\}$ & $67$ & $\{x^2+4x+3, x^2+3x+3\}$ \\
$23$ & $x^2+2x+4$ & $71$ & $x^2+5x+1$ \\
$29$ & $x^2+6x+1$ & $73$ & $x^2+4x+3$ \\
$31$ & $\{x^2+5x+4, x^2+2x+4\}$ & $79$ & $\{x^2+6x+5,x^2+x+5\}$ \\
$37$ & $x^2+4x+2$ & $83$ & $x^2+5x+1$ \\
$41$ & $x^2+2x+1$ & $89$ & $x^2+5x+4$ \\
$43$ & $x^2+6$    & $97$ & $x^2+6$ \\
$47$ & $x^2+2x+4$ & $101$ & $x^2+3x+2$ \\
\hline
\end{tabular}
\end{center}
\end{table}
\end{center}

For $\ell=7$, the computations above are not needed to prove the
complete decomposition result of the next section.
\cref{alg:tracetest} and \cref{prop:inducedorsixchars} apply directly. 
Their role becomes essential for $\ell=13$, where the hypothesis of
\cref{lem:splittorusnormaliser} is unavailable and we instead use
the projective-order criterion of \cref{lem:projorder}.

\subsection{Complete decomposition for \texorpdfstring{$\ell=7$}{ell=7}}
\label{subsec:results_tracetest}

\begin{thm}
\label{thm:ell7_complete_decomposition}
Let $C$ be one of the $381$ Picard curves considered above, and let 
$\ell=7$. 
Suppose that both \cref{hypothesis1,hypothesis2} hold.
Then $\rhoellKss$ is a direct sum of six characters.
\end{thm}

\begin{proof}
We ran \cref{alg:tracetest} for all $381$ curves with $\ell=7$
and all Galois conjugate pairs (type 
\labelcref{item:conjandualpair}) of characters found by
\cref{alg:find_varepsilon} using Euler factors at the good 
primes $p < 2^{13}$. For every curve in the dataset, the set 
$S$ of bad primes is a subset of $\{2,3,5,7\}$, so $D$ has 
$15$ elements, i.e., there are $15$ quadratic extensions 
$K'=K(\sqrt{d})$ to be considered.

For $342$ of the $381$ curves, \cref{alg:tracetest} eliminates 
each of the $15$ possible values of $d$, thus ruling out 
alternative \labelcref{item:induced} of 
\cref{prop:inducedorsixchars}.

For the curve $C_5$ of \cref{ex:eightcharscurve} and the 
Galois-conjugate pair of characters having no compatible type 
\labelcref{item:4tuple} block, \cref{alg:tracetest} again 
eliminates all the $15$ possible values of $d$. Thus both 
alternatives of \cref{prop:inducedorsixchars} are ruled out in 
this case, which proves that this pair of characters is 
spurious.

For the remaining $39$ curves at least one $d$ survives: 
$d = 3$ for $34$ curves, $d = 7$ for $4$ curves and $d = 21$ 
for $3$ curves, with all three surviving for the single curve
$y^3 = 392x^4 + 112x^3 - 264x^2 + 72x + 18$.
In each of these $39$ cases, considering Euler factors at more primes 
would not help, and the surviving values of $d$ can be explained 
as follows. Write $\eta$ for the quadratic character of
$K'/K$, and $\chi_1,\chi_2,\chi_3,\chi_4$ for the four characters other
than $\varepsilon,\varepsilon^c$ in the six-character combination
returned by \cref{alg:find_varepsilon}. Then we find that, up to 
reindexing, $\chi_2 = \chi_1 \eta$ and $\chi_4 = \chi_3\eta$.

For each surviving extension $K'/K$, we repeat
\cref{alg:find_varepsilon} over $K'$. The resulting list of
candidate $1$-dimensional constituents of
$\rhoell|_{G_{K'}}^{ss}$ consists entirely of restrictions of the
characters already obtained over $K$.
Suppose that $\rho$ were irreducible. By
\cref{prop:inducedorsixchars}, we would have
$\rho\simeq\Ind_{K'}^K\psi$
for one of the surviving extensions $K'/K$. Restricting to $G_{K'}$
gives
$\rho|_{G_{K'}}^{ss}\simeq\psi\oplus\psi^\tau$,
where $\tau$ is the nontrivial element of $\Gal(K'/K)$. Hence $\psi$
must occur in the candidate list over $K'$, so
it must be the restriction of some character $\chi:G_K\to\F_7^\times$. 
But then $\rho \simeq\Ind_{K'}^K\Res_{K'}\chi \simeq\chi\oplus\chi\eta$,
contradicting the irreducibility of $\rho$.

Hence, under the standing \cref{hypothesis1,hypothesis2} (that 
$\rho_\lam$ is reducible and $f_p$ factors completely over $\Fell$ for 
every good prime $p$), this proves that $\rhoellKss$ is a direct sum of 
six characters.
The scripts implementing \cref{alg:tracetest} along with their 
output for all $381$ curves, are available at 
\cite{ReducibleRepsPicardGithub}.
\end{proof}

\begin{remark}
\label{rem:tracetest_conditional}
Making the result unconditional is same as the question of 
verifying that the obtained set of six candidate characters 
$G_K \to \Fell^\times$ all appear in $\rhoellKss$. This can be 
done by computing and verifying equality of Frobenius 
characteristic polynomials up to some large bound, coming 
from an effective Chebotarev density theorem. But this bound 
can be too big to carry out this computation, so we have not
attempted to do this.
\end{remark}

\subsection{The \texorpdfstring{$\ell=13$}{ell=13} examples}
\label{subsec:results_ell13}

For $\ell=13$, the argument of
\cref{subsec:results_tracetest} does not apply directly, since
$13\equiv1\pmod4$. We instead use the determinant and trace
computations of \cref{subsec:results_tracedetrho} to exhibit an
element of sufficiently large projective order and then apply
\cref{lem:projorder}.

The work \cite{CG26} found five Picard curves in Sutherland's
$7$-smooth Picard curve dataset such that the
$13$-torsion representation over $K=\Q(\zeta_3)$ is possibly 
reducible. They are all cubic twists of the curve:
\begin{equation}
\label{eq:nonsurj13curves}
  C_1' \colon y^3 = 243x^4 + 338x^3 - 147x^2 - 387x - 142.
\end{equation}
Since all of them have good reduction at all primes above $7$,
the cubic twists only involve primes up to $7$, and hence do not
affect the shape of the $13$-torsion representation. So it is
enough to consider just one of the curves.

We ran \cref{alg:find_varepsilon} with the Euler factors at 
good primes $p < 2^{13}$. For every good prime considered, the 
characteristic polynomial $f_p$ splits completely into linear 
factors over $\F_{13}$.
\cref{alg:find_varepsilon} returns six characters, occurring 
as three Galois conjugate pairs. Similar to the results for 
$\ell=7$ discussed in \cref{subsec:results_epsilon1}, these six 
characters form two blocks, one of type 
\labelcref{item:conjandualpair} and one of type 
\labelcref{item:4tuple}. So if $\rho_\lam$ is indeed reducible, 
the possibilities for the set of characters appearing in 
$\rhoellKss$ are: the type \labelcref{item:conjandualpair} 
block alone, or its union with the type \labelcref{item:4tuple} 
block. The character $\varepsilon$ appearing in the type 
\labelcref{item:conjandualpair} block is ramified at exactly 
one of the two primes above $13$.

Since $13 \not\equiv 3 \pmod 4$, \cref{lem:splittorusnormaliser} does
not apply here. We use \cref{lem:projorder} in its place, for 
which we need to exhibit an element of projective order at least
$6$ in the image of $\rho$. The order of $\rho(\Frob_\pp)$ in 
$\PGL_2(\F_{13})$ is the order of the ratio of
the two eigenvalues, which is determined by
$\tr(\rho(\Frob_\pp))^2/\det(\rho(\Frob_\pp))$. Note that this 
is insensitive to the sign ambiguity in the trace computation,
mentioned in \cref{prop:method2_returns_atmosttwotraces}. 
Among the $550$ good primes $\pp|p$ of $K$ of residue 
degree one with $p < 2^{12}$, exactly $190$ give 
$\rho(\Frob_\pp)$ of projective order 6. So both 
\cref{lem:projorder} and \cref{prop:inducedorsixchars} 
apply for this curve with $\ell = 13$.

Finally, we run \cref{alg:tracetest}. Since the set $S$ of bad
primes is contained in $\{2,3,5,7\}$, the set $D$ has $31$ elements 
for $\ell=13$. \cref{alg:tracetest} eliminates all $31$ 
corresponding quadratic extensions $K' = K(\sqrt{d})$ of $K$.

\begin{thm}
\label{thm:ell13_complete_decomposition}
Let $C$ be the Picard curve given in \cref{eq:nonsurj13curves}, and 
let $\ell=13$. 
Suppose that both \cref{hypothesis1,hypothesis2} hold.
Then $\rhoellKss$ is a direct sum of six characters.
\end{thm}
\begin{proof}
The computations above exhibit an element of projective order
$6$ in the image of $\rho$. Hence
\cref{lem:projorder,prop:inducedorsixchars} apply. The
$31$ quadratic extensions allowed by
\cref{prop:inducedorsixchars} are all eliminated by
\cref{alg:tracetest}. Thus the irreducible alternative is
impossible, and $\rhoellKss$ is a direct sum of six characters.
\end{proof}

\section{Further questions}
\label{sec:further_questions}

The main results of this paper are conditional on
\cref{hypothesis1,hypothesis2}: reducibility of
the $\lam$-torsion representation $\rho_\lam$ and 
splitting of all characteristic polynomials $f_p$ mod $\ell$. 
The computations give strong evidence for these hypotheses in the 
examples considered.
These results suggest the following more general question.

\begin{question}
Let $C/\Q$ be a Picard curve, let $\ell\equiv1\pmod3$, and suppose
that $\rho_\lam$ admits a $1$-dimensional subrepresentation for some
prime $\lam$ of $K$ above $\ell$. 
Can the complementary $2$-dimensional quotient be irreducible?
If so, what additional hypothesis will force reducibility?
\end{question}

\begin{question}
  If it is possible for the complementary subquotient to be 
  irreducible, can we produce such an example? Can its projective 
  image be non-dihedral?
\end{question}

Recall that by \cref{prop:descent}, if an absolutely irreducible 
$2$-dimensional quotient of $\rho_\lam$ exists, then it 
descends, up to twist, to a representation of $G_\Q$ 
whose parity is unconstrained. In that case:
\begin{question}
  Can we produce an example
  of a Picard curve $C/\Q$, where this parity is even?
\end{question}

The initial motivation behind this project was to 
produce new empirical evidence for Serre's modularity 
conjecture over $K=\Q(\zeta_3)$ through such irreducible
$2$-dimensional quotients. 
The first empirical evidence in \cite{Fig99} comes from considering 
restrictions to small imaginary quadratic fields of certain even 
$\SL_2(\F_q)$-valued representations of $G_\Q$ for small $q$. 
Figuieredo's method of producing such representations hinges 
on exploiting exceptional isomorphisms of the groups 
$\PSL_2(\F_q)$ and $\PGL_2(\F_q)$ with $A_n$ and $S_n$, 
and so works only for small prime powers $q$. 

On the other hand, reducible representations of Picard curves 
is a more natural source for $\GL_2(\F_\ell)$-valued 
representations of $K=\Q(\zeta_3)$.
When the curve is defined not over $\Q$, but only over $K$, 
the representations $\rho_\lam$ and $\rho_{\lambar}$ are no 
longer related by Galois conjugation. So whenever $\rho_\lam$
is reducible, one expects to obtain an irreducible $2$-dimensional
subquotient representation $\rho$. But there is no reason to expect it 
to descend to $\Q$.
There is a concrete roadmap to produce such representations: generate 
a dataset of $B$-smooth Picard curves over $K$ (and not over $\Q$) 
for a small $B$ by following \cite{Sut19}; identify the curves in the 
resulting dataset for which the $\lam$-torsion representation of the 
Jacobian is possibly reducible. The algorithms of \cite{CG26}
can already do the second step. The methods developed in 
\cref{subsec:varepsilon,subsec:detrho,subsec:tracerho} can then
determine Frobenius characteristic polynomials of the $2$-dimensional
subquotient $\rho$, which will be particularly useful for proving
its modularity.

Finally, the methods developed in this article work more generally for 
any abelian threefold whose geometric endomorphism algebra is 
isomorphic to an imaginary quadratic field $K$, as long as $\ell$ 
splits in $K$. It should be possible to extend the techniques 
to deal with inert primes. There are several such examples in 
\cite{CG26}, and they are good candidates for further exploration.

\printbibliography
\end{document}